\documentclass[a4paper,11pt,oneside]{amsart}

\usepackage[cyr]{aeguill}
\usepackage[english]{babel}
\usepackage{tikz,hyperref}

\usepackage[T1]{fontenc}
\usepackage{amsmath,bbold}
\usepackage{amsfonts}
\usepackage{amssymb}
\usepackage{amsthm,eepic}
\usepackage{mathrsfs}
\usepackage{enumitem}
\usepackage{graphicx}
\usepackage{footnote}
\usepackage{hyperref,tikz}
\allowdisplaybreaks

\renewcommand{\geq}{\geqslant}
\renewcommand{\leq}{\leqslant}
\renewcommand{\epsilon}{\varepsilon}

\hypersetup{colorlinks=true,    linkcolor=MyDarkBlue,
citecolor=BrickRed,       filecolor=BrickRed,       urlcolor=darkgreen
}
\usepackage{color}
\definecolor{darkgreen}{rgb}{0,0.4,0}
\definecolor{MyDarkBlue}{rgb}{0,0.08,0.85}
\definecolor{BrickRed}{rgb}{0.8,0.08,0}

\numberwithin{equation}{section}

\colorlet{darkgreen}{green!50!black}

\newtheoremstyle{sans}{\parskip}{\parskip}{\itshape}
                       {0pt}{\bfseries\sffamily}{.}{ }{}
\newtheoremstyle{sansplain}{\parskip}{\parskip}{}
                       {0pt}{\bfseries\sffamily}{.}{ }{}

\theoremstyle{sans}
\newtheorem{prop}{Proposition}[section]
\newtheorem{coro}[prop]{Corollary}
\newtheorem{thm}[prop]{Theorem}
\newtheorem{lem}[prop]{Lemma}

\theoremstyle{sansplain}
\newtheorem{rem}[prop]{Remark}

\numberwithin{equation}{section}
\numberwithin{figure}{section}

\usepackage{epsfig}
\usepackage{dsfont}
\def\egaldef{\stackrel{\textnormal{\tiny def}}{=}}

\usepackage{geometry}
\usepackage{mathtools}
\usepackage{eqnarray,cases}

\DeclareMathOperator{\sgn}{sgn}

\newcommand\Cc{\mathbb{C}}

\let\phi=\varphi
\let\epsilon=\varepsilon
\def\DD{\displaystyle}
\def\H{\mathcal{H}}
\def\K{\mathcal{K}}
\def\D{\mathcal{D}}

\def\U{\mathcal{U}}

\begin{document}

\title[Reflected Brownian motion in a non-convex wedge]{On the stationary distribution of reflected Brownian motion in a non-convex wedge}

\date{\today}

\author{Guy Fayolle}

\address{Inria Paris, 2 rue Simone Iff, CS 42112, 75589 Paris Cedex 12 France \newline\indent
Inria Paris-Saclay, 1 rue Honor\'e d'Estienne d’Orves, 91120 Palaiseau, France} 
\email{guy.fayolle@inria.fr}

\author{Sandro Franceschi}
       
        \address{T\'el\'ecom SudParis, Institut Polytechnique de Paris, 19 place Marguerite Perey, 91120 Palaiseau, France
        } \email{sandro.franceschi@telecom-sudparis.eu}

\author{Kilian Raschel}
       
        \address{Universit\'e d'Angers, Laboratoire Angevin de Recherche en Math\'ematiques, CNRS, 2~Boulevard Lavoisier, 49000 Angers, France} \email{raschel@math.cnrs.fr}
        \thanks{This project has received funding from two organizations: \\ (1) The European Research Council (ERC) under the European Union's Horizon 2020 research and innovation programme under the Grant Agreement No.\ 759702. \\ (2) The ANR RESYST (ANR-22-CE40-0002)}
        


\begin{abstract}
We study the stationary reflected Brownian motion in a non-convex wedge, which, compared to its convex analogue model, has been much rarely analyzed in the probabilistic literature. We prove that its stationary distribution can be found by solving a two dimensional vector boundary value problem (BVP) on a single curve for the associated Laplace transforms. The reduction to this kind of vector BVP seems to be new in the literature. As a matter of comparison, one single boundary condition is sufficient in the convex case. When the parameters of the model (drift, reflection angles and covariance matrix) are symmetric with respect to the bisector line of the cone, the model is reducible to a standard reflected Brownian motion in a convex cone. Finally, we construct a one-parameter family of distributions, which surprisingly provides, for any wedge (convex or not), one particular example of stationary distribution of a reflected Brownian motion. 
\end{abstract}

\maketitle

\section*{To the memory of Vadim Malyshev}  On September 30, 2022, at the age of 85, Vadim Aleksandrovich Malyshev, Editor-in-Chief of the journal~MPRF, died suddenly.  Vadim was an outstanding Russian scientist in the field of probability and mathematical physics. His memory will always remain in the hearts and minds of his colleagues. I [Guy Fayolle]  mourn the loss of the one who was my friend for 37 years.

\section{Introduction} \label{sec:RBM}

\subsection{Context and motivations} Since the introduction of the reflected Brownian motion in the eighties \cite{HaRe-81a,HaRe-81b,VaWi-85,Wi-85}, the mathematical community has shown a constant interest in this topic. Typical questions deal with  the recurrence of the process, the absorption at the corner of the wedge, the existence and computation of stationary distributions... We refer for more details to the introduction of \cite{FrRa-19}. 

Generally speaking, an obliquely reflected Brownian motion in a two-dimensional wedge of opening angle $\beta\in (0, 2\pi)$ is defined by its drift $\mu\in\mathbb R^2$ and two reflection angles $(\delta,\varepsilon)\in(0, \pi)^2$, see Figures \ref{fig:Franceschi_condition2}, \ref{fig:three_quarter_plane} and \ref{fig:symmetric_case_change_variable} for a few examples. The covariance matrix is taken to be the identity. A suitable linear transform allows to reduce the whole range of parameter angles $\beta\in (0, 2\pi)$ to only three cases: the quarter plane (when $\beta\in (0,\pi)$), the three-quarter plane (when $\beta\in (\pi,2\pi)$) and the limiting half-plane case $\beta=\pi$. Doing so, the covariance matrix is nolonger the identity but instead has the general form \eqref{eq:covariance_matrix}. However, by a clear convexity argument, a linear transform cannot be used to transform, for instance, the three-quarter plane into a quarter plane.

While the early articles \cite{VaWi-85,Wi-85} most dealt with the general case $\beta\in (0, 2\pi)$ (see also the more recent article \cite{KaRa-14}), the subcase of convex cones $\beta\in (0, \pi]$ has attracted much more attention \cite{HaRe-81a,HaRe-81b,Fo-82,Fo-84,BaFa-87,DiMo-09,DaMi-11,DaMi-13,franceschi_tuttes_2016,FrRa-19,BoElFrHaRa-21}; we have identified at least three reasons for that. First, one initial motivation was to approximate queueing systems in a dense traffic regime \cite{Ha-78}, which are typically obtained from random walks in the (convex) quarter plane. Second, the Laplace transform turns out to be a very useful tool in these problems; to make this function converge we need to have a convex cone. Finally, because there are already several parameters defining reflected Brownian motion (drift, reflection angles and opening angle), we feel that non-convex cones have sometimes been taken away, in order to reduce the number of cases to consider: for instance, regarding transience and recurrence criteria, only the convex case has been established in \cite{HoRo-93}, while close arguments should also cover the non-convex case.

\begin{quote}
\textit{In this article, our main objective is the study of recurrent reflected Brownian motion in the non-convex case $\beta\in (\pi,2\pi)$: we shall introduce complex analysis techniques to characterize the Laplace transform of the stationary distribution.}
\end{quote}

Let us present five motivations to the present work. Our first goal is to complete the literature and to show how, in this more complicated non-convex setting, one can solve the problem of finding the stationary distribution. Our techniques could also be applied to the transient case, for example to analyse Green functions or absorption probabilities (see \cite{Fr-21,ErFr-21} for the convex case); however, we do not tackle these problems here.

Our second motivation is provided by the discrete framework of random walks (or queueing networks). Indeed, in the same way as in the quarter plane, reflected Brownian motion has been introduced to study scaling limits of large queueing networks (see Figure~\ref{fig:intro_scaling-1}), a Brownian model in a non-convex cone could approximate discrete random walks on a wedge having obtuse angle (see Figure \ref{fig:intro_scaling-2} for a concrete example). Such random walks have an independent interest and have already been studied in a number of cases: see \cite{BM-16,RaTr-19,EP-20} in the combinatorial literature and \cite{Tr-19,Mu-19} for more probability inclined works. 

\begin{figure}[hbtp]
\centering
\includegraphics[scale=0.5]{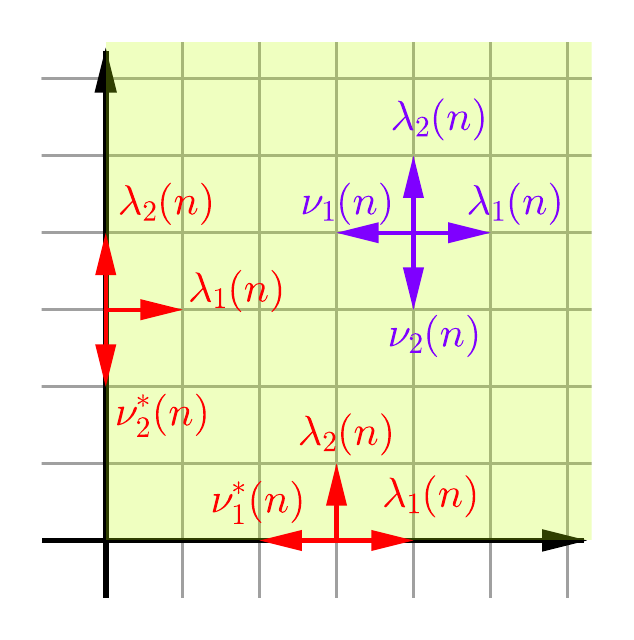}
\includegraphics[scale=0.5]{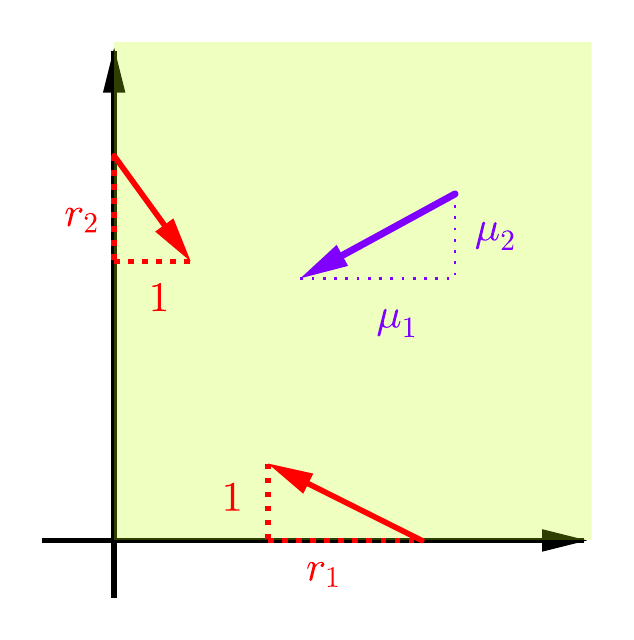}
\caption{Scaling limit of some queueing systems towards reflected Brownian motion. Left picture: transition rates of a random walk (two coupled processors). Taking $\lambda_{i}(n),\nu_i(n)\to\frac{1}{2}$, $\sqrt{n}(\lambda_i-\nu_i)\to \mu_i$ and $\nu_i^*(n)\to\frac{r_i+1}{2}$, the discrete process converges to the reflected Brownian motion with parameters described as on the right picture (with identity covariance matrix). See \cite{Re-84} for the original proof.}
\label{fig:intro_scaling-1}
\end{figure}

\begin{figure}[hbtp]
\centering
\includegraphics[scale=0.5]{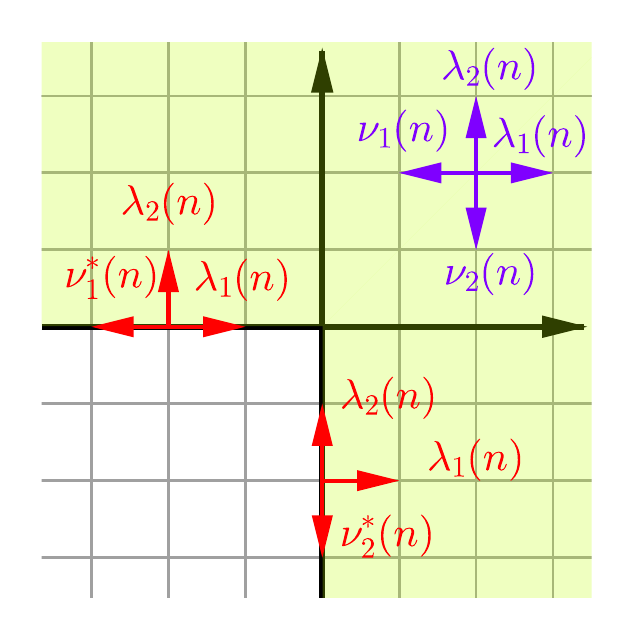}
\includegraphics[scale=0.5]{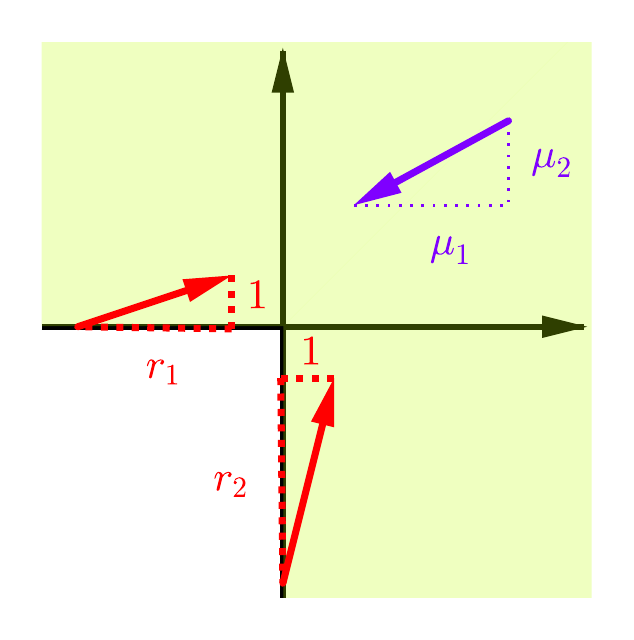}
\caption{For the exact same reasons as for Figure \ref{fig:intro_scaling-1}, in the three-quarter plane, the discrete model on the left picture converges to the reflected Brownian motion on the right display.}
\label{fig:intro_scaling-2}
\end{figure}

Our third motivation is to develop an analytic method, which turns out to be particularly useful in a number of contexts. This method was invented by Fayolle, Iasnogorodski and Malyshev in the seventies, see \cite{Ma-71,FaIa-79,FIM-2017}; at that time, the principal motivation was to study the stationary distribution of ergodic reflected random walks in a quadrant. The main idea is to state a functional equation satisfied by the associated generating functions and to reduce it to certain boundary value problems, which after analysis happen to be solvable in closed form. This approach has been applied to the framework of Brownian diffusions in a quadrant \cite{Fo-82,Fo-84,BaFa-87}, to symmetric random walks in a three-quarter plane \cite{RaTr-19,Tr-19}, but never to the present setting of diffusions in non-convex wedges. From this technical point of view, the present work will bring the following novelty: we will prove that our problem is generically reducible to a system of two boundary value problems (as a matter of comparison, only one single boundary value problem is needed in the convex case \cite{FrRa-19}). This formally leads to a matrix power series for the Laplace transform, as a solution of a Fredholm integral equation, see \eqref{eq:equaint}.

\begin{figure}[hbtp]
\centering
\includegraphics[scale=0.4]{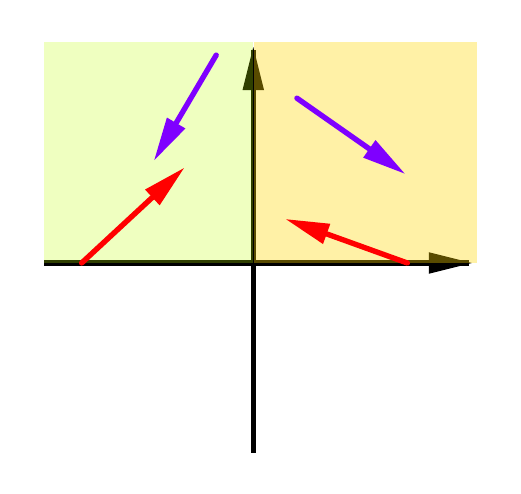}
\includegraphics[scale=0.4]{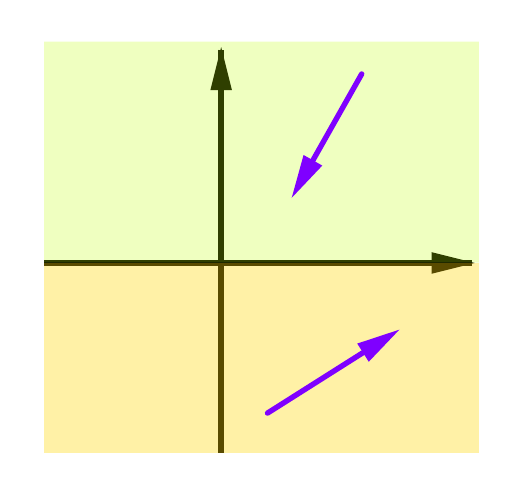}
\includegraphics[scale=0.4]{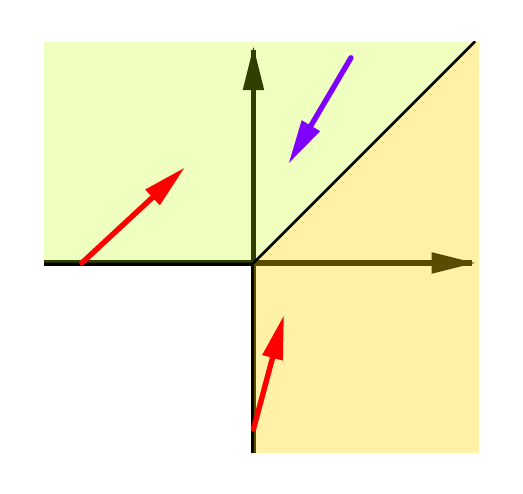}
\includegraphics[scale=0.4]{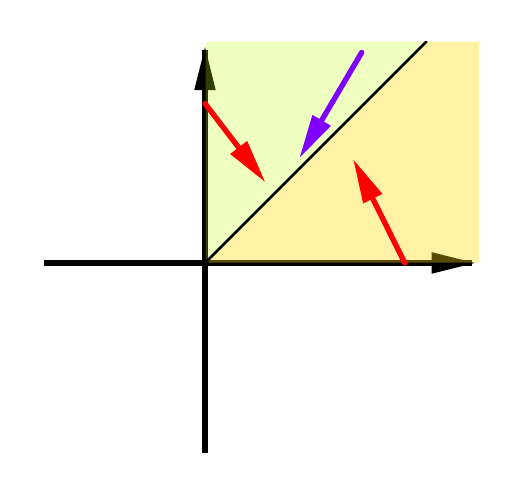}
\caption{Different models of (non-)reflected inhomogeneous Brownian motions in various cones of $\mathbb R^2$. Blue arrows represent drift vectors and red arrows stand for the reflection vectors on the boundary axes. In the second picture, when the two drifts are opposite, the vertical component is called a bang-bang Brownian motion.}
\label{fig:intro}
\end{figure}

Next, we aim at initiating the study of piecewise inhomogeneous Brownian models in cones of $\mathbb R^d$. To take a concrete example, consider a half-plane and view it as the union of two quarter planes glued along one half-axis (see Figure \ref{fig:intro}, leftmost picture). Then the process behaves as follows: in each quarter plane, its evolution is governed by a Brownian motion (with possible different drifts and covariance matrices); the process can pass from one quadrant to the other one through the porous interface; on the remaining boundaries, it is reflected in a standard way. Another example would consist in dividing the plane into two half-planes, as on Figure \ref{fig:intro}, left. This model may be viewed as a two-dimensional generalization of the so-called bang-bang process on $\mathbb R$, as studied in \cite{Sh-81}.

Piecewise inhomogeneous Brownian motions are related to our obtuse angle model as follows: splitting the three-quarter plane into two convex wedges (see the right display on Figure \ref{fig:intro}, or Figure \ref{fig:split}) and performing simple linear transformations, our model turns out to be equivalent to the inhomogeneous domain described above. 

These inhomogeneous models are reminiscent from a well-known model in queuing theory, known as the JSQ (for ``join the shortest queue'') model, see \cite[Chap.~10]{FIM-2017} or \cite{KuSu-03}. In this model, the quarter plane is divided into two octants ($\pi/8$-wedges) and the random walk obeys to different (very specific rules) according to the octant. See the rightmost picture on Figure \ref{fig:intro}. The techniques developed in this paper offer a potential approach to solve this (asymmetric) Brownian JSQ model.

Our fifth and final motivation is to provide tools leading to a comparative study of reflected Brownian motion in convex and non-convex cones. Does this model admit a kind of phase transition around the critical angle $\beta=\pi$? Some results in our paper tend to show that this is the case: while reflected Brownian motion in a convex cone may be studied with one single boundary value problem, two analogue problems are needed in the non-convex case. On the other hand, we also bring some evidence that the model has a smooth behavior at $\beta=\pi$: we are able to construct a one-parameter family of stationary distributions, whose formula is valid for any $\beta\in(0,2\pi)$ and, surprisingly, is independent of $\beta$! While we will leave the question of phase transition as an open problem, let us conclude with the expression of the density (written in polar coordinates) of this remarkable family:
\begin{equation}
\label{eq:expression_remarkable_density}
   \pi(r,t)=\frac{C}{\sqrt{r}} \cos\Bigl(\frac{t}{2}\Bigr) e^{-2r \vert \mu\vert \cos^2\left(\frac{t}{2}\right)},\quad  \vert t\vert\leq \frac{\beta}{2}<\pi,
\end{equation}
where $\vert\mu\vert$ stands for the norm of the drift and $C$ is a  normalization constant; see Figure~\ref{fig:Franceschi_condition2}. The example \eqref{eq:expression_remarkable_density} is obtained  \cite{BoElFrHaRa-21} in the convex case, it immediately extends to the non-convex case.

\begin{figure}[hbtp]
\centering
\includegraphics[scale=0.4]{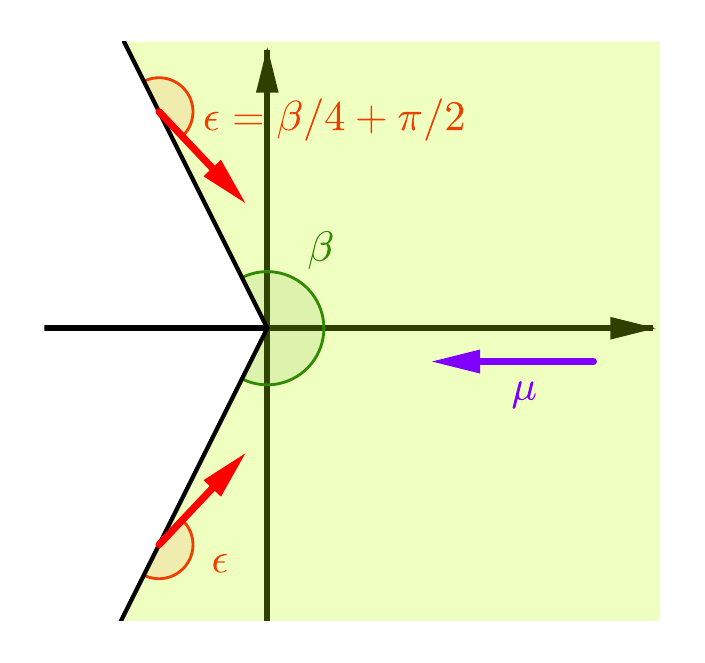}
\includegraphics[scale=0.4]{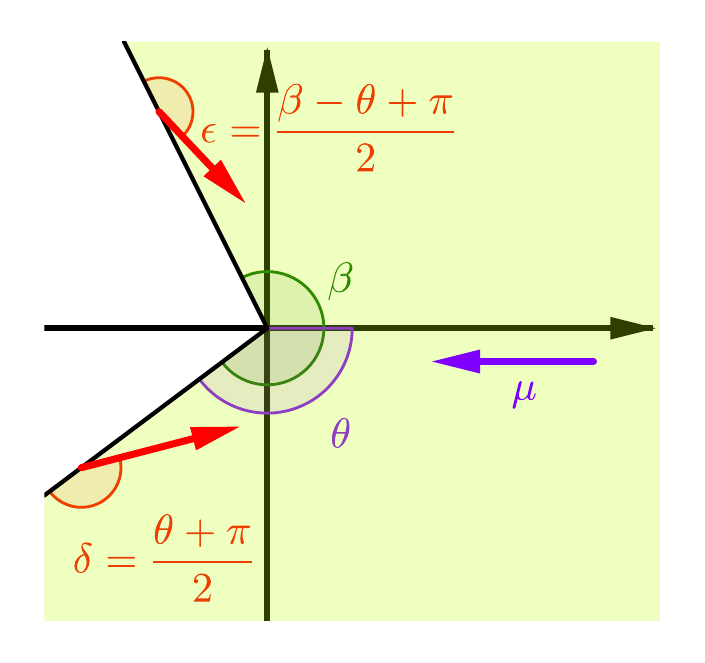}
\caption{Parameters of the model leading to the remarkable~stationary distribution \eqref{eq:expression_remarkable_density}. A priori, no symmetry assumption is done on the parameters (the model on the left is symmetric, contrary to the one on the right). Similarly, no convex hypothesis is done on the cone. The formula \eqref{eq:expression_remarkable_density} has been obtained in \cite{BoElFrHaRa-21} in the convex case and in \cite[\S 9]{Ha-78} in a more restrictive case, and we observe here that the same formula holds for any value of the opening angle $\beta$. Up to our knowledge, \eqref{eq:expression_remarkable_density} is the unique example for which the stationary distribution density is known in closed form for a non-convex cone.}
\label{fig:Franceschi_condition2}
\end{figure}

\subsection{Main results}
To conclude this introduction, we present the structure of the paper and our main results.
\begin{itemize}
   \item Section \ref{sec:model}: definition of the model, statement of the recurrence conditions and introduction of the stationary distribution, Proposition \ref{prop:BAR} on the classical basic adjoint relationship (characterizing the stationary distribution)
   \item Section \ref{sec:functional_equation}: Proposition \ref{prop:main_func_eq} on a system of two functional equations (the $3/4$ plane is split into two convex cones of angle $3\pi/8$, and one equation is stated for each domain) 
   \item Section \ref{sec:asymmetric_case}: general study of the asymmetric case. Various statements on the kernel, meromorphic continuation of the unknown Laplace transforms, reduction to a Riemann-Hilbert vector boundary value problem (Theorem~\ref{thm:vectorial_BVP}), relation with a Fredholm integral equation
   \item Section \ref{sec:symmetric_case}: general study of the symmetric case. Equivalence with a standard Brownian motion in a quarter plane, resolution and examples
\end{itemize}

\subsection*{Acknowledgments}
We thank Andrew Elvey Price and Kavita Ramanan for interesting discussions. 
\section{Semimartingale reflected Brownian motion avoiding a quarter plane}
\label{sec:model}

\subsection{Definition of the process}

We denote the three-quarter plane as
\begin{equation*}
   S \egaldef \{ (z_1,z_2)\in \mathbb{R}^2 : z_1 \geq 0 \text{ or } z_2 \geq 0 \}.
\end{equation*}
The parameters of the model are the drift $\mu=(\mu_1,\mu_2)$, the reflection vectors $R_1=(r_1,1)$ and $R_2=(1,r_2)$, and the covariance matrix
\begin{equation}
\label{eq:covariance_matrix}
\Sigma= \left(
\begin{array}{cc}
\sigma_1 & \rho \\ 
\rho & \sigma_2
\end{array} 
\right),
\end{equation}
see Figure \ref{fig:three_quarter_plane}. Throughout this study, $\Sigma$ will be assumed to be elliptic, i.e., $\sigma_1\sigma_2-\rho^2>0$, thus discarding the degenerated case $\sigma_1\sigma_2-\rho^2=0$.
\begin{figure}[hbtp]
\centering
\includegraphics[scale=0.4]{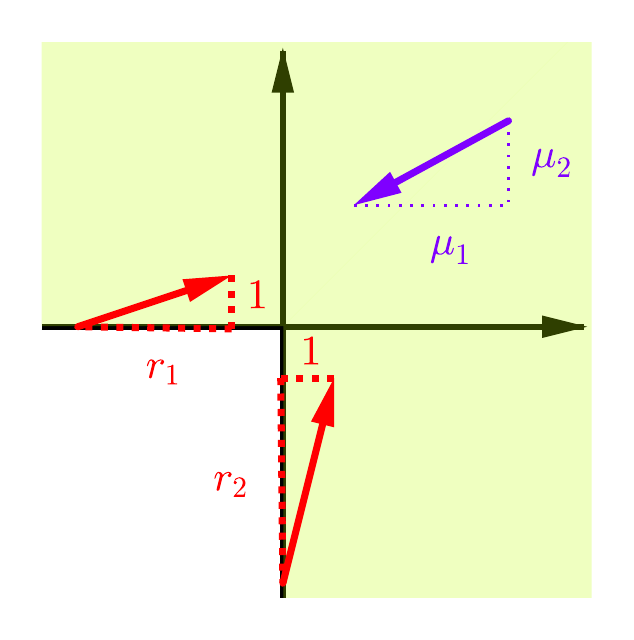}
\caption{In green color, the three-quarter plane $S$, in blue the drift $\mu$ and in red the reflection vectors $R_1$ and $R_2$.}
\label{fig:three_quarter_plane}
\end{figure}

More specifically, we define the obliquely reflected Brownian motion $Z_t=(Z_t^1,Z_t^2)$ in the three-quarter plane $S$ as follows:
\begin{equation}
\label{eq:defZt}
\begin{cases}
Z_t^1\egaldef Z_0^1+ W_t^1+ \mu_1 t+ r_1 L_t^1 + L_t^2,
\\
Z_t^2\egaldef Z_0^2+ W_t^2+ \mu_2 t+ L_t^1 + r_2 L_t^2,
\end{cases}
\end{equation}
where $W_t$ is a planar Brownian motion of covariance $\Sigma$, $L^1_t$ is (up to a constant) the local time on the negative part of the abscissa ($z_1\leq 0$) and $L_t^2$ is the local time on the negative part of the ordinate axis ($z_2\leq 0$). In case of a zero drift, such a semimartingale definition of reflected Brownian motion is proposed in the reference paper \cite{Wi-85} (including the non-convex wedges); it readily extends to our drifted case.

Throughout this paper, we assume that the process is positive recurrent and has a unique stationary distribution (or invariant measure). As it turns out, this is equivalent to
\begin{equation}
\label{eq:CNS1_stationary}
   \mu_1<0 \quad \text{and} \quad \mu_2<0,
\end{equation}
together with
\begin{equation}
\label{eq:CNS2_stationary}
   \mu_1-r_1\mu_2>0 \quad \text{and} \quad
   \mu_2-r_2\mu_1>0.
\end{equation}
(In particular, one has $r_1>0$ and $r_2>0$.) We couldn't find any reference proving this statement; however, the same techniques as in \cite{HoRo-93} by Hobson and Rogers or \cite[Sec.~6]{HaWi-87} (proving necessary and sufficient conditions in the quadrant similar as \eqref{eq:CNS1_stationary} and \eqref{eq:CNS2_stationary}) could be used here. Figure \ref{fig:three_quarter_plane} represents a case where the parameters satisfy both conditions \eqref{eq:CNS1_stationary} and \eqref{eq:CNS2_stationary}. 
The heuristic of these conditions is the following. The process is either recurrent or transient, and if the process is transient, then it tends to infinity. By \eqref{eq:CNS1_stationary}, the drift vector is negative and there are only two possible behaviours for the process to tend to infinity: either, as $t\to\infty$, $Z_t^1$ tends to $-\infty$ and $Z_t^2 \geqslant 0$, or $Z_t^2$ tends to $-\infty$ and $Z_t^1 \geqslant 0$. So, we come down to a couple of problems in half-planes, which are easy to understand, since reflected Brownian motion in a half-plane is a well-studied process. For example, in the upper half-plane, the conditions for the process $Z_t^1$ not to tend to $-\infty$ is $\mu_1-r_1\mu_2 \geqslant 0$ ($\mu_1-r_1\mu_2=0$ is a null recurrent  case). Combining the two conditions leads heuristically to~\eqref{eq:CNS2_stationary}. Indeed, coupling arguments associated with a pathwise construction could make the above reasoning more rigourous, but we shall omit them.

Under conditions \eqref{eq:CNS1_stationary} and \eqref{eq:CNS2_stationary}, we denote by $\Pi$ the unique stationary distribution. In the case of a quarter plane, it is proved in \cite{HaWi-87} that $\Pi$ admits a density with respect to the Lebesgue measure, see Lemma~12 in \cite[Sec.~7]{HaWi-87}. Using exactly the same argument (in particular Lemma~9 in \cite[Sec.~7]{HaWi-87}), we deduce that in the three-quarter plane, $\Pi$ admits a density, which we will denote by $\pi$. We also define the boundary invariant measures by
\begin{equation*}
   {\nu}_{1} (A) = \mathbb{E}_\Pi \int_0^1 \mathrm{1}_{A\times\{0\} } (Z_s) \mathrm{d}L_s^1
   \quad \text{and} \quad
   {\nu}_{2} (A) = \mathbb{E}_\Pi \int_0^1 \mathrm{1}_{\{0\} \times A} (Z_s) \mathrm{d}L_s^2.
\end{equation*}
The measure ${\nu}_{1}$ has its support on $\{z_1 \leq 0 \}$ and ${\nu}_{2}$ has its support on $\{z_2 \leq 0 \}$. We will also denote by $\nu_1 (z_1)$ and  $\nu_2 (z_2)$ their respective densities.

Remark that a reflected Brownian motion in the three-quarter plane could be defined as well in the non-semimartingale case; motivations to consider these cases are proposed in \cite{RaRe-03}.

\subsection{Basic adjoint relationship}
Our approach is based on the following identity, called basic adjoint relationship, which in the orthant case is proved in \cite{DaHa-92,HaWi-87}.

\begin{prop}
\label{prop:BAR}
If $f$ is the difference of two convex functions in $S$, and if $\int_S f(z_1,z_2)\pi(z_1,z_2) \mathrm{d}z_1 \mathrm{d}z_2$ and all the integrals below converge, then
\begin{multline*}
\int_S \mathcal{G} f(z_1,z_2) \pi(z_1,z_2) \mathrm{d}z_1 \mathrm{d}z_2 +\\ \int_{-\infty}^0 R_1 \cdot \nabla f(z_1,0) \nu_1 (z_1) \mathrm{d}z_1
+
\int_{-\infty}^0 R_2 \cdot \nabla f(0,z_2) \nu_2 (z_2)\mathrm{d}z_2=0,
\end{multline*}
where the generator is equal to
\begin{equation*}
\mathcal{G} f = \frac{1}{2} \left(\sigma_1 \frac{\partial^2 f}{\partial z_1^2}+2\rho \frac{\partial^2 f}{\partial z_1 \partial z_2} +\sigma_2 \frac{\partial^2 f}{\partial z_2^2}
 \right)
+ \mu_1 \frac{\partial f}{\partial z_1}
+\mu_2 \frac{\partial f}{\partial z_2}.
\end{equation*}
\end{prop}
\begin{proof}
We apply the It\^o-Tanaka formula to the semimartingale $Z_t$, see Theorem~1.5 in \cite[Chap.~VI \S 1]{Revuz1999}. Note that, in the formula of the previous reference, there is no need to assume that $f$ is $\mathcal{C}^2$ since, when $f$ is convex, its second derivative in the sense of distibution is a positive measure. We obtain
\begin{equation*}
   f(Z_t)=f(Z_0)+\int_0^t \mathcal{G}f (Z_s) \mathrm{d}s +\int_0^t \nabla f (Z_s) \cdot \mathrm{d} W_s +  \sum_{i\in\{1,2\}} \int_0^t R_i \cdot \nabla f (Z_s) \mathrm{d} L_s^i .
\end{equation*}
To conclude, we take the expectation over $\Pi$ in the above equality.
\end{proof}
Since we take $f$ to be the difference of two convex functions, the first derivatives of $f$ are defined as the left derivatives, and the second derivatives of $f$ are understood in the sense of distributions.
 
\begin{rem}
\label{rem:diff}
Continuity and differentiability of the measure $\pi(z_1,z_2)$ directly follow from the properties of weak solutions to the 
partial differential equation satisfied by $\pi$ and stated in Proposition~\ref{prop:BAR}.  Indeed, a famous result known as Weyl's lemma \cite[Lem.~2]{We-40}
asserts that a weakly harmonic function coincides almost everywhere with a strongly harmonic function, and is in particular smooth. This result generalizes to distributions associated to hypoelliptic operators. Here, $\int_S \mathcal{G} f(z_1,z_2) \pi(z_1,z_2) \mathrm{d}z_1 \mathrm{d}z_2 =0$, for all $f$ which are smooth in $S$ and which cancel near the boundary of $S$, and we deduce that $\pi$ is smooth inside of $S$.
\end{rem}

\section{The main functional equations}
\label{sec:functional_equation}


Our goal is to use the basic adjoint relationship of Proposition~\ref{prop:BAR} to obtain a kernel equation for the Laplace transform of the stationary distribution. In the case of a convex cone, it is enough to take $f(z_1,z_2)=e^{x z_1+ y z_2}$ to obtain the functional equation, see \cite{DaMi-11,franceschi_tuttes_2016,FrRa-19}. However, if the cone is not convex, the associated integrals will not converge. So we need to divide the three-quarter plane into  two regions. 
We define the two following $\frac{3}{8}$-planes:
\begin{equation*}
   S_1\egaldef\{(z_1,z_2)\in\mathbb{R}^2: z_1\leq z_2 \text{ and } z_2 \geq 0 \}
\end{equation*}
 and $S_2\egaldef S\setminus S_1$, see Figure~\ref{fig:split}.
 
\begin{figure}[hbtp]
\centering
\includegraphics[scale=0.6]{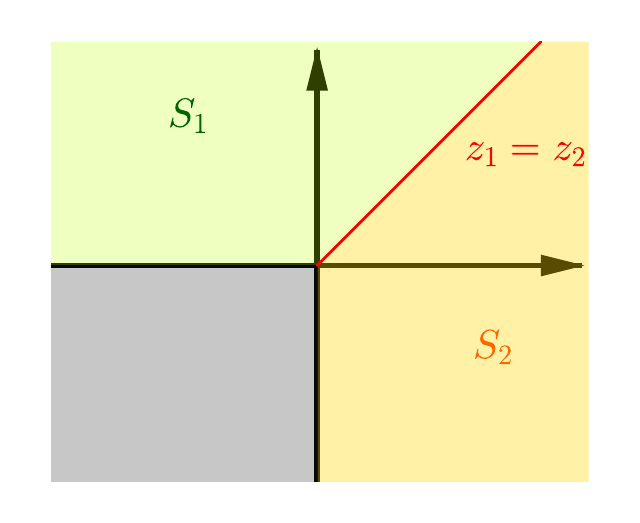}\quad
\includegraphics[scale=0.73]{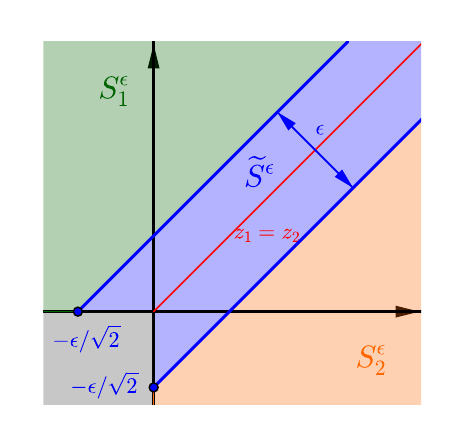}
\caption{Left: the three-quarter plane divided in two sets, $S_1$ in green and $S_2$ in orange. Right: the three sets $S_1^\epsilon$ (in green color), $S_2^\epsilon$ (orange) and $\widetilde{S}^\epsilon$ (blue).}
\label{fig:split}
\end{figure}

Let us define the Laplace transform of the invariant measure $\pi$ in $S_1$ by
\begin{equation}
\label{eq:formula_Laplace_transform_L1}
L_1(x,y)
\egaldef \int_{S_1} e^{x z_1+ y z_2} \pi(z_1,z_2) \mathrm{d}z_1 \mathrm{d}z_2,
\end{equation}
the Laplace transform of $\pi$ on the diagonal 
\begin{equation*}
m(x+y) \egaldef \int_0^\infty e^{(x+y) z} \pi(z,z) \mathrm{d}z ,
\end{equation*}
the Laplace transform of the normal derivative of $\pi$ on the diagonal (which does exist by Remark~\ref{rem:diff})
\begin{equation}
\label{eq:formula_Laplace_transform_n}
 n(x+y) \egaldef  \int_0^\infty e^{(x+y) z}
 \left( \frac{\partial \pi}{\partial z_1}(z,z) -\frac{\partial \pi}{\partial z_2}(z,z)
 \right)
 \mathrm{d}z,
 \end{equation}
and the Laplace transform of the boundary measure $\nu_1$ on the abscissa
\begin{equation*}
\ell_1(x) 
 \egaldef \int_{-\infty}^0 e^{x z_1}
\nu_1 (z_1) \mathrm{d}z_1.
\end{equation*} 
Introduce finally the constant 
\begin{equation*}
\theta \egaldef \frac{\sigma_1+\sigma_2 -2\rho}{2},
\end{equation*}
which is positive due to the ellipticity condition $\rho^2-\sigma_1\sigma_2<0$.

The remainder of Section \ref{sec:functional_equation} is devoted to proving the following result:
\begin{prop}[Functional equation in $S_1$]
\label{prop:main_func_eq}
For all $(x,y)$ in $\{\Re{(x)}\geq0,\,\Re{(x+y)}\leq0\}$, we have
\begin{multline*}
  -K(x,y)L_1(x,y) =\\ k (x,y) m (x+y) +\theta n(x+y) + k_1(x,y) \ell_1(x)
+(1- r_1) \nu_1(0) + (r_2- 1) \nu_2(0),
\end{multline*}
where the kernel is defined by
\begin{equation}
\label{eq:kernel_K}
   K(x,y)\egaldef \frac{1}{2} \left(\sigma_1 x^2 +2\rho xy +\sigma_2 y^2\right)+ \mu_1 x+\mu_2 y,
\end{equation}
while $k$ and $k_1$ are polynomials of degree one in two variables given by
\begin{align*}
k (x,y) &\egaldef 
\frac{\theta(y-x)}{2}+ \frac{1}{2}(\sigma_2-\sigma_1)(x+y) +\mu_2-\mu_1,
\\
k_1(x,y) &\egaldef r_1 x +y.
\end{align*}
\end{prop}
A symmetric functional equation holds on the domain $S_2$; it involves the functions $m$ and $n$ above, as well as
\begin{equation*}
L_2(x,y)
\egaldef \int_{S_2} e^{x z_1+ y z_2} \pi(z_1,z_2) \mathrm{d}z_1 \mathrm{d}z_2
\quad
\text{and}
\quad
\ell_2(y) 
 \egaldef \int_{-\infty}^0 e^{y z_2} \nu_2 (z_2) \mathrm{d}z_2.
\end{equation*} 
See \eqref{eq:sys2} for the exact statement. The proof of Proposition~\ref{prop:main_func_eq} is rather lengthy and postponed to Appendix~\ref{app:p}.

\section{The general asymmetric case}
\label{sec:asymmetric_case}

\subsection{Sketch of the approach}
For the sake of brevity, we shall put
\begin{equation*}
E \egaldef (1- r_1) \nu_1(0) - (1-r_2) \nu_2(0).
\end{equation*}
Then the two functional equations obtained in Section~\ref{sec:functional_equation} (see in particular Proposition~\ref{prop:main_func_eq}), corresponding to the regions $S_1$ and $S_2$ in the $(z_1,z_2)$-plane, see Figure \ref{fig:split}, are simply rewritten as follows:
\begin{equation}\label{eq:sys1}
K(x,y)L_1(x,y)+ k (x,y) m (x+y) + \theta n(x+y) + k_1(x,y) \ell_1(x) + E = 0,
\end{equation}
\hspace{2cm} \emph{in the region} $\{\Re{(x)}\geq0,\,\Re{(x+y)}\leq0\}$;
\begin{equation}\label{eq:sys2}
K(x,y)L_2(x,y) - k(x,y) m (x+y) - \theta n(x+y) + k_2(x,y) \ell_2(y) - E = 0, 
\end{equation}
\hspace{2cm} \emph{in the region} $\{\Re{(y)}\geq0,\,\Re{(x+y)}\leq0\}$.

\medskip
\noindent The main idea is to build a system, where the new variables are defined in \emph{one and the same region}, by means of a simple change of variables. Clearly, this operation has a cost, since there will be two different kernels, the positive side being they can be simultaneously analyzed  starting from a common domain.
The key milestones of  the study are listed hereunder:
\emph{\begin{itemize}
   \item Make the meromorphic continuation of all functions in their respective (cut) complex planes (see Theorem~\ref{thm:cont}).
   \item Construct a vectorial Riemann boundary value problem for the pair $(\ell_1,\ell_2)$ (see Theorem~\ref{thm:vectorial_BVP}).
   \item Derive a Fredholm integral equation for $m$ (see Equation~\eqref{eq:equaint}). 
\end{itemize}}

\subsection{Functional equations and kernels}\label{sec:equafunc}
Setting respectively 
\begin{equation*}
   p=-x, \quad  q= x+y, \quad \mbox{in Equation \eqref{eq:sys1}},
\end{equation*}
and
\begin{equation*}
   p=-y, \quad  q= x+y, \quad \mbox{in Equation \eqref{eq:sys2}},
\end{equation*}
leads to the system 
\begin{align}
U(p,q)L_1(p,q)+ A(p,q) m(q) + \theta n(q) + C(p,q) \ell_1(p) + E &= 0,\label{eq:sys3}\\[0.2cm]
V(p,q)L_2(p,q)+ B(p,q) m(q) - \theta n(q) + D(p,q) \ell_2(p) - E &= 0, \label{eq:sys4}
\end{align}
where both equations are a priori defined in the domain $\{\Re{(p)}\leq 0,\,\Re{(q)}\leq0\}$, and
\begin{equation}\label{eq:Ker1} 
\begin{cases}
U(p,q) \egaldef \DD \theta p^2 +  \frac{\sigma_2}{2} q^2 +(\sigma_2-\rho)pq +(\mu_2-\mu_1)p+ \mu_2 q, \\[0.3cm] 
V(p,q) \egaldef \DD \theta p^2  +  \frac{\sigma_1}{2} q^2 +(\sigma_1-\rho)pq +(\mu_1-\mu_2)p+ \mu_1 q,\\[0.3cm]
\DD A(p,q) \egaldef \frac{\theta (2p+q)}{2}+ \frac{(\sigma_2-\sigma_1)q}{2} +\mu_2-\mu_1, \\[0.2cm]
\DD B(p,q) \egaldef \frac{\theta (2p+q)}{2}+ \frac{(\sigma_1-\sigma_2)q}{2} +\mu_1-\mu_2, \\[0.2cm]
C(p,q) \egaldef (1-r_1)p+q, \\[0.2cm]
D(p,q) \egaldef (1-r_2)p+q.
\end{cases}
\end{equation}

\paragraph{\textbf{Notation}}
\emph{For convenience and to distinguish between the two kernels, we shall add in a superscript position the letter $u$ (resp.\ $v$)\ to any quantity related to the kernel $U(x,y)$ (resp.\ $V(x,y)$). 
Moreover, if a property holds both for $u$ and $v$, the superscript letter is omitted ad libitum}. 

\medskip

Accordingly, the branches of the algebraic curve $U=0$ (resp.\ $V=0$) over the $q$-plane will be denoted by $P^u_i(q)$ (resp.\ $P^v_i(q)$), $i=1,2$. By definition, they are solutions to
\begin{equation}
\label{eq:def_U_P_V_P}
   U(P_i^u(q),q)=0\quad \text{and}\quad V(P_i^v(q),q)=0.
\end{equation}
In particular, they are simple algebraic functions of order $2$. Similarly, $Q^u_i(p)$ (resp.\ $Q^v_i(p)$) will stand for the branches over the $p$-plane, $i=1,2$. 

Although we are mostly working under the stationary hypotheses \eqref{eq:CNS1_stationary} and \eqref{eq:CNS2_stationary}, notice that Lemmas~\ref{lem:PuPv} and \ref{lem:QuQv} below hold true for any value of the drift vector $(\mu_1,\mu_2)$.

\begin{lem}
\label{lem:PuPv}
The functions $P^u_i(q)$ and $P^v_i(q)$, $i=1,2$, are analytic in the whole complex plane cut along $(-\infty,q_1]\cup [q_2,\infty)$, where the branch points $q_1<0$ and $q_2>0$ are the two real roots of the equation 
\begin{equation}
\label{eq:bpP}
   (\rho^2-\sigma_1\sigma_2)q^2 + 2[\mu_1(\rho-\sigma_2) + \mu_2(\rho-\sigma_1)]q 
+(\mu_1-\mu_2)^2 = 0.
\end{equation}
Remarkably, $q_1$ and $q_2$ are the same for the two kernels $U$ and $V$. Moreover:
\begin{itemize}
   \item The branches $P_1^u$ and $P_2^u$ are separated and satisfy 
\begin{equation}\label{eq:sep-pu}
\begin{cases}
\DD \Re(P_1^u(ix)) \leq 0 \leq \Re(P_2^u(ix)),  \quad  \forall x\in \mathbb{R}, \\[0.2cm]
\DD \Re(P_1^u(q))  \leq \Re(P_2^u(q)), \quad \forall q \in \Cc,  \\[0.2cm]
\DD P_1^u(0) = \min\left\{0, \frac{\mu_1-\mu_2}{2\theta}\right\}, \quad
P_2^u(0) = \max\left\{0, \frac{\mu_1-\mu_2}{2\theta}\right\}.
\end{cases}
\end{equation}
They map the cut $(-\infty,q_1]$ (resp.\ $[q_2,\infty)$) 
onto  the right branch $\H^u_+$  (resp.\ left branch $\H^u_-$) of the hyperbola $\H^u$ with equation
\begin{equation}
\label{eq:hpu}
   (\rho^2 -\sigma_1\sigma_2)x^2 + (\sigma_2 - \rho)^2 y^2 +2(\sigma_2\mu_1-\rho\mu_2)x +
\frac{(\mu_2 - \mu_1)(\sigma_2(\mu_1 + \mu_2) - 2\rho\mu_2)}{2\theta} = 0.
\end{equation}
\item 
 The branches $P_1^v$ and $P_2^v$ are separated and satisfy 
\begin{equation}\label{eq:sep-pv}
\begin{cases}
\DD \Re(P_1^v(ix)) \leq 0 \leq \Re(P_2^v(ix)),  \quad  \forall x\in \mathbb{R}, \\[0.2cm]
\DD \Re(P_1^v(q))  \leq \Re(P_2^v(q)), \quad \forall p \in \Cc, \\[0.2cm]
\DD P_1^v(0) = \min\left\{0, \frac{\mu_2-\mu_1}{2\theta}\right\}, \quad
P_2^v(0) = \max\left\{0, \frac{\mu_2-\mu_1}{2\theta}\right\}.
\end{cases}
\end{equation}
They map the cut $(-\infty,q_1]$ (resp.\ $[q_2,\infty)$) 
onto  the right branch $\H^v_+$ (resp.\ left branch $\H^v_-$) of the hyperbola $\H^v$ with equation
\begin{equation}\label{eq:hpv}
(\rho^2-\sigma_1\sigma_2)x^2 + (\sigma_1 - \rho)^2 y^2 +2(\sigma_1\mu_2-\rho\mu_1)x +
\frac{(\mu_1 - \mu_2)(\sigma_1(\mu_1 + \mu_2) - 2\rho\mu_1)}{2\theta} = 0.
\end{equation}
\end{itemize}
\end{lem}

\begin{figure}[hbtp]
\centering
\includegraphics[scale=1.5]{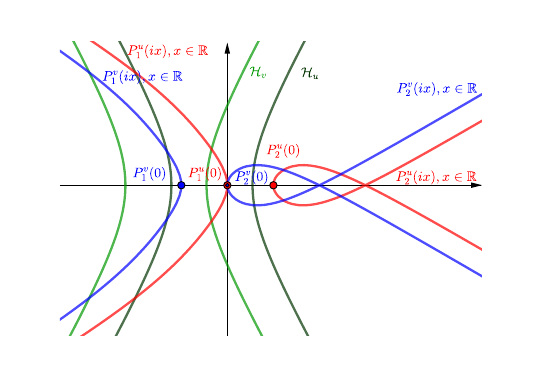}
\caption{Illustration of Lemma~\ref{lem:PuPv} in the case where $\mu_1>\mu_2$: curves $\{ P_1^u(i x) : x\in\mathbb{R}\}$ and $\{ P_2^u(i x) : x\in\mathbb{R}\}$ in red; curves $\{ P_1^v(i x) : x\in\mathbb{R}\}$ and $\{ P_2^v(i x) : x\in\mathbb{R}\}$ in blue; hyperbolas $\mathcal{H}_u$ and $\mathcal{H}_v$ in dark green and light green respectively. Playing with the parameters is possible thanks to the following GeoGebra animation \href{https://www.geogebra.org/m/phvjk35w}{www.geogebra.org/m/phvjk35w}}
\label{fig:dessinlemmePuPv}
\end{figure}

\begin{proof} 
The branch points of $P(q)$ are the zeros of the discriminant of $U(p,q)=0$ viewed as a polynomial in $p$, and equation \eqref{eq:bpP} follows directly. 

In order to prove \eqref{eq:sep-pu}, let $P(q)$ denote the multivalued algebraic function satisfying \eqref{eq:def_U_P_V_P}. Letting $q=ix$ with $x\in\mathbb{R}$ and $P(q)\egaldef\alpha+i\beta$ with real $\alpha,\beta$, then separating real and imaginary parts, we obtain
\begin{equation}\label{eq:sep2}
\begin{cases}
\DD \theta\alpha^2 + (\mu_2-\mu_1)\alpha - \bigl(\theta\beta^2 + (\sigma_2-\rho)x\beta + \frac{\sigma_2}{2}x^2\bigr) =0, 
\\[0.2cm]
\DD \beta(2\theta\alpha +\mu_2-\mu_1) + x\bigl(\alpha(\sigma_2-\rho)+\mu_2\bigr)=0.
 \end{cases}
\end{equation}
Then one checks that the first equation of \eqref{eq:sep2}, viewed as a polynomial in $\alpha$, has two real roots with opposite sign. Indeed, the quadratic polynomial in~$\beta$
\begin{equation*}
\theta\beta^2 + (\sigma_2-\rho)x\beta + \frac{\sigma_2}{2}x^2
\end{equation*}
is always positive, due to the ellipticity condition.

The second property of \eqref{eq:sep-pu} is a direct application of the maximum modulus principle applied to the function $\exp P(q)$. More precisely, we look at the function $\exp P(q)$ on the domain $\mathbb C\setminus ((-\infty,q_1]\cup [q_2,\infty))$. Using the first property of \eqref{eq:sep-pu}, we deduce that for some values of $q$, one has 
\begin{equation}
\label{eq:ineq_P1_P2}
   \bigl\vert \exp P_1(q)\bigr\vert\leq\bigl\vert \exp P_2(q)\bigr\vert.
\end{equation}
On the other hand, on the cut $q\in(-\infty,q_1]\cup [q_2,\infty)$, the branches $P_1(q)$ and $P_2(q)$ are complex conjugate and thus $\bigl\vert \exp P_1(q)\bigr\vert=\bigl\vert \exp P_2(q)\bigr\vert$. Since the cut is the boundary of the cut plane, the maximum modulus principle entails that the inequality \eqref{eq:ineq_P1_P2} holds true globally on $\mathbb C$.

The analytic expression \eqref{eq:hpu} of the hyperbola follows from direct computations, see Lemma \ref{lem:hpq} and its proof for similar computations.

We note the pleasant symmetry of \eqref{eq:bpP} with respect to the parameters. This is mainly due to the change of parameters from $(x,y)$ to $(p,q)$. As it will emerge later, that symmetry plays an important role in our analysis. The proof of the lemma is complete.
\end{proof}

Quite analogous properties hold for $Q^u_i(p)$ and $Q^v_i(p)$, but now the branch points depend on the kernel. They are partially listed in the next lemma, where the equations of the hyperbolas are 
omitted.

\begin{lem}
\label{lem:QuQv}
The functions $Q^u_1(p)$ and $Q^u_2(p)$ are analytic in the complex plane cut along $(-\infty,p_1^u]\cup [p_2^u,\infty)$, where the branch points $p^u_1<0$ and $p^u_2>0$ are the real roots of the equation
\begin{equation}
\label{eq:Qu}
   (\rho^2-\sigma_1\sigma_2)p^2 + 2(\sigma_2\mu_1-\rho\mu_2)p + \mu_2^2 = 0.
\end{equation}
The branches $Q_1^u$ and $Q_2^u$ are separated and satisfy 
\begin{equation}\label{eq:sep-qu}
\begin{cases}
\DD \Re(Q_1^u(ix)) \leq  0  \leq \Re(Q_2^u(ix)),  \quad  \forall x\in \mathbb{R}, \\[0.2cm]
\DD \Re(Q_1^u(p))   \leq  \Re(Q_2^u(p)), \quad \forall p \in \Cc, \\[0.2cm]
\DD Q_1^u(0)= \min\left\{0,\frac{-2\mu_2}{\sigma_2}\right\}, \quad 
Q_2^u(0)= \max\left\{0,\frac{-2\mu_2}{\sigma_2}\right\}.
\end{cases}
\end{equation}
They map  the cut $(-\infty,p_1^u]$ (resp.~$[p_2^u,\infty)$) onto the right branch $\K_+^u$ (resp.\ the left branch $\K_-^u$) of the hyperbola $\K^u$.

Similarly, the functions $Q^v_1(p)$ and $Q^v_2(p)$ are analytic in the complex plane cut along $(-\infty,p_1^v]\cup [p_2^v,\infty)$, where the branch points $p^v_1<0$ and $p^v_2>0$ are the real roots of the equation
\begin{equation}
\label{eq:Qv}
   (\rho^2-\sigma_1\sigma_2)p^2 + 2(\sigma_1\mu_2-\rho\mu_1)p + \mu_1^2 = 0.
\end{equation}
The branches $Q_1^v$ and $Q_2^v$ are separated and satisfy 
\begin{equation}\label{eq:sep-qv}
\begin{cases}
\DD \Re(Q_1^v(ix)) \leq  0  \leq \Re(Q_2^v(ix)),  \quad  \forall x\in \mathbb{R}, \\[0.2cm]
\DD \Re(Q_1^v(p))   \leq  \Re(Q_2^v(p)), \quad \forall p \in \Cc, \\[0.2cm]

\DD Q_1^v(0)= \min\left\{0,\frac{-2\mu_1}{\sigma_1}\right\}, \quad 
Q_2^v(0)= \max\left\{0,\frac{-2\mu_1}{\sigma_1}\right\}.
\end{cases}
\end{equation}
They map  the cut $(-\infty,p_1^v]$ (resp.~$[p_2^v,\infty)$) onto the right branch $\K_+^v$ (resp.\ the left branch $\K_-^v$) of the hyperbola $\K^v$.
\end{lem}

It is worth remarking at once that, by using \eqref{eq:sys3}, \eqref{eq:sys4} and Lemma~\ref{lem:PuPv}, one can set \emph{two  boundary value problems}  for the couple of functions $[\ell_1(p),\ell_2(p)]$ on the respective hyperbolas~$\H_+^u$ and $\H_+^v$.

\subsection{Meromorphic continuation to the complex plane}

The method relies on an iterative algorithm, as in \cite[Chap.~10]{FIM-2017}, and the following theorem holds. Below and throughout, if $\H_\pm$ denotes a branch of hyperbola as on Figure \ref{fig:dessinlemmePuPv}, $\H_{\pm,\textnormal{int}}$ will represent the left connected component of $\mathbb C\setminus \H_{\pm}$. 

\begin{thm}
\label{thm:cont}
The functions $m$, $n$, $\ell_1$ and $\ell_2$ can be continued as meromorphic functions to the whole complex plane cut along proper positive real half-lines in their respective planes. The number of poles is finite, and the possible poles of $m$ and $n$ inside the domain $\K_{-,\textnormal{int}}^u \bigcup \H_{-,\textnormal{int}}^v$ coincide.
\end{thm}

\subsection{Reduction to a vectorial Hilbert boundary value problem} \label{sec:VBVP}

For all $q\in(-\infty,q_1]$, Equations \eqref{eq:sys3} and \eqref{eq:sys4} yield the linear system
\begin{equation*}
\begin{cases}
A(P_1^u(q),q) m(q) + \theta n(q) + C(P_1^u(q),q) \ell_1(P_1^u(q)) + E = 0, 
\\[0.2cm]
B(P_1^v(q),q) m(q) - \theta n(q) + D(P_1^v(q),q) \ell_2(P_1^v(q)) - E = 0,
\end{cases}
\end{equation*}
which  in turn gives
\begin{equation}\label{eq:mn}
\begin{cases}
 m(q) &\hspace{-2mm}=\  \DD \frac{C(P_1^u(q),q) \ell_1(P_1^u(q)) + D(P_1^v(q),q) \ell_2(P_1^v(q))}{\Delta(q)}, \\[0.3cm]
n(q) &\hspace{-2mm}=\  \DD \frac{B(P_1^v(q),q) C(P_1^u(q),q)\ell_1(P_1^u(q)) - 
 A(P_1^u(q),q) D(P_1^v(q),q) \ell_2 (P_1^v(q))}{\theta\Delta(q)} - \frac{E}{\theta},
\end{cases}
\end{equation}
where 
\begin{equation}\label{eq:delta}
\Delta(q) \egaldef  -  \bigl(A(P_1^u(q),q) + B(P_1^v(q),q)\bigr) =
-\theta\bigl(q+ P_1^u(q) + P_1^v(q)\bigr).
\end{equation}
Now, by using the continuity of the left-hand side of the system \eqref{eq:mn} when $q$ traverses the cut $(-\infty,q_1]$, we can set a two-dimensional homogeneous Hilbert boundary value problem for the vector $[\ell_1,\ell_2]$. More precisely, we first deduce from \eqref{eq:mn} the two following relations, which hold for all $q\in(-\infty,q_1]$:
\begin{multline}
\label{eq:bvpm}
\frac{C(P_1^u(q),q) \ell_1(P_1^u(q)) + D(P_1^v(q),q) \ell_2(P_1^v(q))}{\Delta(q)} =\\
\frac{C(\overline{P_1^u(q)},q) \ell_1(\overline{P_1^u(q)}) + D(\overline{P_1^v(q)},q) \ell_2(\overline{P_1^v(q)})}{\overline{\Delta(q)}},
\end{multline}
and
\begin{multline}
 \lefteqn{\frac{C(P_1^u(q),q) B(P_1^v(q),q)\ell_1(P_1^u(q)) -  A(P_1^u(q),q) D(P_1^v(q),q) \ell_2 (P_1^v(q))}
 {\Delta(q)}} \\
 = \frac{B(\overline{P_1^v(q)},q) C(\overline{P_1^u(q)},q)\ell_1(\overline{P_1^u(q)}) -  
 A(\overline{P_1^u(q)},q) D(\overline{P_1^v(q)},q) \ell_2 (\overline{P_1^v(q)})}
 {\overline{\Delta(q)}}. \label{eq:bvpn}
\end{multline}

Introducing the vector $L(q)\egaldef [\ell_1(P_1^u(q)), \ell_2(P_1^v(q))]$ and the $2\times2$-matrix
\begin{equation}
   G(q)\egaldef \frac{1}{\overline{\Delta(q)}}
\begin{pmatrix}\label{eq:mat}
&\DD \frac{-\bar{\gamma}(\alpha+\bar{\beta)}}{\gamma}
&\DD \frac{\bar{\delta}(\bar{\alpha}-\alpha)}{\gamma}&\\[0.2cm]
&\DD \frac{\bar{\gamma}(\bar{\beta}-\beta)}{\delta}
&\DD \frac{-\bar{\delta}(\beta+\bar{\alpha})}{\delta}&
\end{pmatrix},
\end{equation}
with
\begin{equation}
\label{eq:notation_matrix_G}
   \alpha = A(P_1^u(q),q), \quad \beta = B(P_1^v(q),q), \quad \gamma  =  C(P_1^u(q),q), \quad 
\delta =  D(P_1^v(q),q),
\end{equation}
the system \eqref{eq:bvpm}--\eqref{eq:bvpn} immediately yields the following result:
\begin{thm}
\label{thm:vectorial_BVP}
We have
\begin{equation*} 
L^+(q)=G(q)L^-(q), \quad \forall q\in (-\infty,q_1],
\end{equation*}
where $L^+(q)$ (resp.\ $L^-(q)$) is the limit of $L(q)$  when $q$ reaches the cut from below (resp.\ from above) in the complex plane.
\end{thm}

\begin{rem}
With the notation \eqref{eq:notation_matrix_G}, the determinant of the matrix in \eqref{eq:mat} can be rewritten as 
\begin{equation*}
   \frac{\overline{\gamma \delta}}{\gamma \delta}\frac{\alpha+\beta}{\overline{\alpha+\beta}}.
\end{equation*}
 Its modulus is one, and it is interesting to ask whether this fact could be anticipated. 
\end{rem}

Let us denote by $\omega_u$ the conformal mapping of $\H_{+,\textnormal{int}}^u$ onto the unit disk $\D$. Then $\omega_u$ is analytic in $\H_{+,\textnormal{int}}^u$, its inverse function $\omega_u^{-1}$ is analytic in $\D$, and we have
\begin{equation*}
   \vert \omega_u(p)\vert= 1, \quad \forall p\in\H_+^u.
\end{equation*}
Actually, $\omega_u$ has a known explicit form (see, e.g., Chapter~6 in \cite{Ne-75}). Moreover, by symmetry, one can choose
$\omega_u(\overline{p})=\overline{\omega_u(p)},\ \forall p\in\H_+^u$, so that
\begin{equation*}
\omega_u(\overline{p}) = \frac{1}{\omega_u(p)}.
\end{equation*}
In other words, for $\vert z\vert=1$, we have $\overline{z}=1/z$  
and $\overline{p}=\omega^{-1}_u(1/z)$. Similar definitions hold by exchanging the roles of $u$ and $v$. 

Then, setting
\begin{equation*}
\varPhi^+(z)\egaldef [\ell_1(\omega_u^{-1}(z)), \ell_2(\omega_v^{-1}(z))], \quad \forall z\in\D,
\end{equation*}
and 
\begin{equation*}
\varPhi^-(z)\egaldef \varPhi^+(1/z),  \quad \forall \vert z\vert >1,
\end{equation*}
we obtain the boundary condition 
\begin{equation}\label{eq:VBC}
   \varPhi^+(z)= H(z)\varPhi^-(z), \quad \forall \vert z\vert=1,
\end{equation}
where $H(z)$ is the $2\times2$ matrix directly derived from $G(q)$, given in~\eqref{eq:mat}, by using the functions $\omega_u(p)$ and $\omega_v(p)$.
The problem can now be formulated as follows: 
\begin{quote} \emph{Find a sectionally meromorphic vector $\varPhi(z)$, constant at infinity, equal to 
 $\varPhi^+(z)$ (resp.\ $\varPhi^-(z)$) for $z\in\D$ (resp.\ for $z\notin\D$), and which satisfies the boundary condition \eqref{eq:VBC}.}
 \end{quote}
 
\subsection{On the solvability of the vectorial boundary value problem \eqref{eq:VBC}}
It is natural to ask whether the boundary value problem \eqref{eq:VBC} may be solved in closed form. As a matter of comparison, scalar (i.e., one-dimensional) boundary value problems may be solved in terms of contour integrals, involving conformal mappings or uniformization techniques. This is the situation encountered in the convex case \cite{BaFa-87,FrRa-19} as well as in the non-convex symmetric case, as shown in the following Section~\ref{sec:symmetric_case}. However, vectorial boundary value problems are in general hardly solvable in closed form \cite{MUS,Vekua-1967,FaRa-15}. 

 Here, after eliminating the possible poles of $\varPhi$  inside the unit disk, the solution to \eqref{eq:VBC} is shown to be directly connected with  the Fredholm integral equation (see, e.g., \cite{MUS,Vekua-1967})
\begin{equation}\label{eq:equaint}
   \varPhi^-(z_0) -\frac{1}{2\pi}\int_{\vert z\vert =1} \frac{H^{-1}(z_0)H(z) - \mathcal{I}}{z-z_0} \varPhi^-(z)\mathrm{d}z = \varPhi^-(\infty),
\end{equation}
where $\mathcal{I}$ stands for the identity matrix. Since all elements of the matrix $H(z)$ are explicitly known, we can express the formal solution of the BVP \eqref{eq:VBC} as a convergent matrix power series
from \eqref{eq:equaint}.

Let us do three additional remarks.
\begin{itemize}
  \item To the best of our knowledge, the only asymmetric case which admits a density in closed form is the one mentioned at the end of the introduction, with explicit formula \eqref{eq:expression_remarkable_density}, see Figure~\ref{fig:Franceschi_condition2}. This example, which works for any opening angle $\beta\in(0,2\pi)$, is not obtained as a consequence of the vectorial problem \eqref{eq:VBC}, but rather from an analogy with the convex case studied in \cite{BoElFrHaRa-21}. However, by a direct (but tedious) algebra, it can be checked a posteriori that the vectorial boundary value problem \eqref{eq:VBC} is satisfied by the solution \eqref{eq:expression_remarkable_density}.
   \item The solvability of \eqref{eq:VBC} should be strongly related to potential nice factorizations of the matrix $H(z)$. For example, in case the matrix $H(z)$ could be written as the product of matrices $\Psi^+(z)^{-1}\Psi^-(z)$, with $\Psi$ sectionally meromorphic on the complex plane cut along the unit circle, then \eqref{eq:VBC} could be rewritten as the homogeneous problem $(\Psi\Phi)^+(z)= (\Psi\Phi)^-(z)$, which is solvable. Finding such factorizations appears as a kind of vectorial Tutte's invariant method, in the terminology of \cite{franceschi_tuttes_2016,BoElFrHaRa-21}.
   \item In the symmetric case, the vectorial problem becomes solvable, as we will see in the next Section~\ref{sec:symmetric_case}. On the other hand, in the non-symmetric case, our work appeals further developments. In this respect, an interesting intermediate semi-symmetrical situation takes place when $\mu_1=\mu_2$, $\sigma_1=\sigma_2$, but $r_1\neq r_2$, which should lead  to some reasonably explicit results.
\end{itemize}

 
\section{The symmetric case}
\label{sec:symmetric_case}

When the model is symmetric, we shall put
\begin{equation*}
   \mu\egaldef\mu_1=\mu_2, \quad \sigma \egaldef\sigma_1=\sigma_2 \quad \text{and}\quad r\egaldef r_1=r_2.
\end{equation*}
The invariant measure is symmetric w.r.t.\ the diagonal $z_1=z_2$. Consequently, we have $\pi (z_1,z_2)=\pi(z_2,z_1)$, which yields $n(x+y)=0$, see \eqref{eq:formula_Laplace_transform_n}.

\begin{figure}[hbtp]
\centering
\includegraphics[scale=0.4]{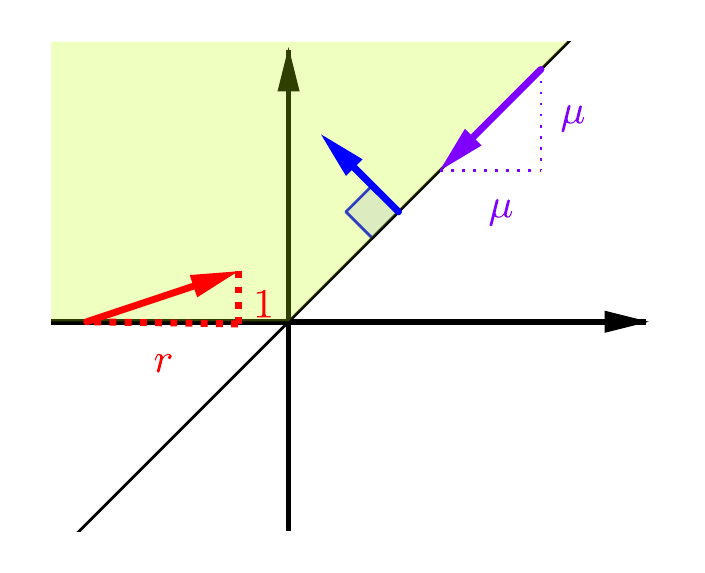}
\includegraphics[scale=0.4]{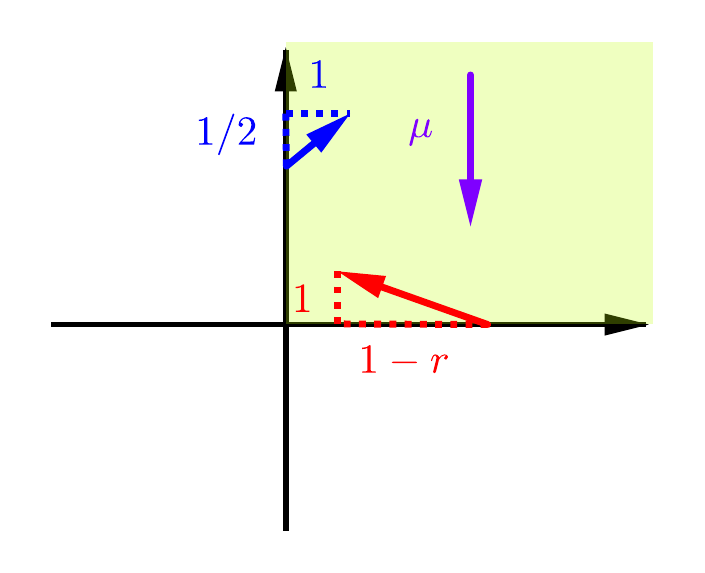}
\includegraphics[scale=0.4]{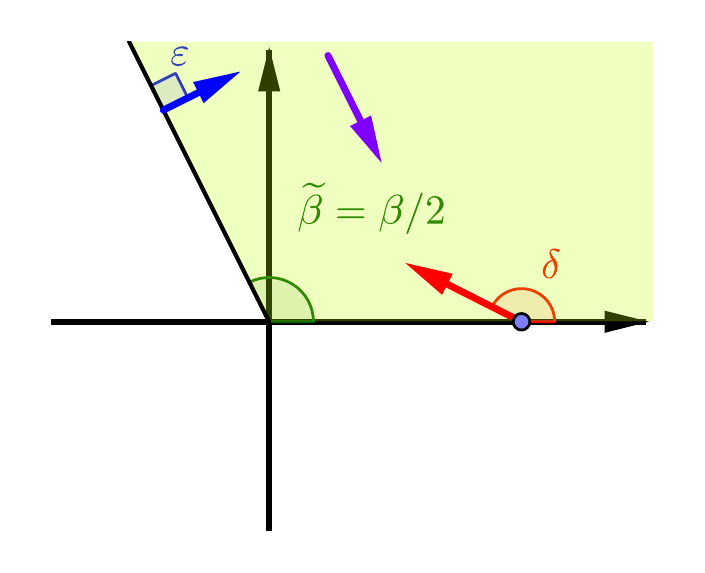}
\caption{In the symmetric case, Brownian motion in a three-quarter plane can be reduced to a more standard reflected Brownian motion in a convex cone. More precisely, its projection in $S_1$ defines a Brownian motion $\widehat{Z}$ in a wedge of opening $3\pi/8$. The left picture above represents the drift and reflection vectors of $\widetilde{Z}$, to be studied in Section~\ref{sec:3/8}. After a first change of variables, it becomes a Brownian motion $\widetilde{Z}$ in the quarter plane, as studied in Section~\ref{sec:3/4}, see the middle picture. On the right, the model is mapped to a $\beta/2$-cone through a linear transform, so as to admit an identity covariance matrix: this last model will be denoted by $T\widetilde{Z}$, see Section~\ref{sec:beta}.}
\label{fig:symmetric_case_change_variable}
\end{figure}

\subsection{Reformulation as a reflected Brownian motion in a $3/8$ plane}
\label{sec:3/8}

Let $\widehat{Z}_t$ be the reflected process of $Z_t$ along the diagonal defined by
\begin{equation*}
\widehat{Z}_t \egaldef (\widehat{Z}_t^1,\widehat{Z}_t^2)  \egaldef
\frac{1}{2}(Z_t^1+Z_t^2-\vert Z_t^2-Z_t^1\vert ,Z_t^1+Z_t^2+\vert Z_t^2-Z_t^1\vert )
 =
\begin{cases}
(Z_t^1,Z_t^2) & \text{if } Z_t \in S_1 ,
\\
(Z_t^2,Z_t^1) & \text{if } Z_t \in S_2  .
\end{cases}
\end{equation*}
As the following result will establish, the process $\widehat{Z}_t$ is a standard reflected Brownian motion in the convex cone $S_1$, with reflection vector $(r,1)$ on the horizontal axis and an orthogonal reflection on the diagonal, see Figure~\ref{fig:symmetric_case_change_variable} (left). We also provide a semimartingale decomposition of this reflected process.
\begin{lem}
\label{lem:semimart38}
In the symmetrical case, we have
\begin{equation*}
\begin{cases}
\widehat{Z}_t^1=\widehat{Z}_0^1+ \widehat{W}_t^1+ \mu t + r \widehat{L}_t^1 - \frac{1}{2}\widehat{L}_t^2,
\\
\widehat{Z}_t^2=\widehat{Z}_0^2+ \widehat{W}_t^2+ \mu t+ \widehat{L}_t^1 +  \frac{1}{2}\widehat{L}_t^2,
\end{cases}
\end{equation*}
where $\widehat{W}_t$ is a Brownian motion with the same covariance matrix as $W_t$, $\widehat{L}_t^2$ is the local time of $\widehat{Z}_t$ on the diagonal, and $\widehat{L}_t^1= L_t^1 + L_t^2$ is the local time of $\widehat{Z}_t$ on the horizontal axis. We deduce that $\widehat{Z}$ is a reflected Brownian motion in a $3/8$-plane, with reflection vector $(r,1)$ on the horizontal axis and an orthogonal reflection on the diagonal.
\end{lem}
\begin{proof}
By \eqref{eq:defZt}, we have
\begin{equation*}
   Z_t^2-Z_t^1=Z_0^2-Z_0^1 +W_t^2-W_t^1 + (r-1)(L_t^2 -L_t^1).
\end{equation*}
We apply Itô-Tanaka formula (see Theorem~1.5 in \cite[Chap.~VI \S 1]{Revuz1999}) to the continuous semimartingale $Z_t^2-Z_t^1$ and to the absolute value $\vert\cdot\vert$. We obtain
\begin{align*}
\vert Z_t^2-Z_t^1\vert
&=Z_0^2-Z_0^1 + \int_0^t \sgn (Z_t^2-Z_t^1) (\mathrm{d} W_t^2 -\mathrm{d} W_t^1) \\ &\quad + (r-1) \int_0^t \sgn (Z_t^2-Z_t^1) (\mathrm{d} L_t^2 -\mathrm{d} L_t^1) + \widehat{L}_t^2
\\ &=Z_0^2-Z_0^1 +\int_0^t \sgn (Z_t^2-Z_t^1) (\mathrm{d} W_t^2 -\mathrm{d} W_t^1)+ (1-r)(L_t^1 +L_t^2)+ \widehat{L}_t^2,
\end{align*}
as $L_t^1$ increases only when ($Z_t^2<0, Z_t^1 =0)$ and $L_t^2$ increases only when  $(Z_t^1<0, Z_t^2 =0).$
Let us recall that, by definition, $\widehat{L}_t^1=L_t^1+L_t^2$. By \eqref{eq:defZt}, we have
\begin{equation*}
   Z_t^1+Z_t^2=Z_0^1+Z_0^2 +W_t^1+W_t^2+2\mu t + (r+1)(L_t^1 +L_t^2).
\end{equation*}
Then, we directly obtain 
\begin{equation*}
\begin{cases}
\widehat{Z}_t^1= \frac{1}{2}(Z_t^1+Z_t^2-\vert Z_t^2-Z_t^1\vert)= \widehat{Z}_0^1+ \widehat{W}_t^1+ \mu t + r \widehat{L}_t^1 - \frac{1}{2} \widehat{L}_t^2,\\
\widehat{Z}_t^2 =\frac{1}{2}(Z_t^1+Z_t^2+\vert Z_t^2-Z_t^1\vert)= \widehat{Z}_0^2+ \widehat{W}_t^2+ \mu t +  \widehat{L}_t^1 + \frac{1}{2} \widehat{L}_t^2,
\end{cases}
\end{equation*}
where we defined
\begin{equation*}
\begin{cases}
\displaystyle\widehat{W}_t^1 \egaldef \int_0^t \frac{1+ \sgn(Z_t^2-Z_t^1)}{2} \mathrm{d}W_t^1 + \int_0^t \frac{1- \sgn(Z_t^2-Z_t^1)}{2} \mathrm{d}W_t^2,\medskip\\
\displaystyle\widehat{W}_t^2 \egaldef \int_0^t \frac{1- \sgn(Z_t^2-Z_t^1)}{2} \mathrm{d}W_t^1 + \int_0^t \frac{1+ \sgn(Z_t^2-Z_t^1)}{2} \mathrm{d}W_t^2.
\end{cases}
\end{equation*}
We easily verify that the associated quadratic variations satisfy $\langle \widehat{W}^1 \rangle_t= \langle W^1 \rangle_t =\sigma_1 t$, $\langle \widehat{W}^2 \rangle_t= \langle W^2 \rangle_t =\sigma_2 t$ and $\langle \widehat{W}^1 , \widehat{W}^2 \rangle_t = \langle W^1, W^2 \rangle_t=\rho t$ and we conclude by Lévy's characterization theorem, see Theorem~3.6 in \cite[Chap.~IV \S 3 p150]{Revuz1999}.
\end{proof}
The reflected process $\widehat{Z}$ is also recurrent and we denote $\widehat{\pi}$ its stationary distribution. 
\begin{prop}
\label{prop:hatpi}
For all measurable sets $A \subset S_1$, we have $\pi(A)=\frac{1}{2} \widehat{\pi} (A)$.
\end{prop}
\begin{proof}
Let $A \subset S_1$ and $\widehat{A}\in S_2$ be the symmetric set with respect to the first diagonal. In the symmetric case, we have
$
\pi(A)=\pi (\widehat{A}).
$
By the ergodic properties of an invariant measure we have $\pi(A) = \lim_{t\to\infty} \mathbb{P}[Z_t \in A]$. Then
\begin{align*} 
\pi(A) &=\frac{1}{2} (\pi(A)+\pi (\widehat{A})) 
\\ &= \frac{1}{2} \lim_{t\to\infty} \left( \mathbb{P}[Z_t \in A] +\mathbb{P}[Z_t \in \widehat{A}] \right)
\\ &= \frac{1}{2} \lim_{t\to\infty} \mathbb{P}[\widehat{Z}_t \in A]
\\ &= \frac{1}{2} \widehat{\pi} (A).\qedhere
\end{align*}
\end{proof}

\subsection{Reformulation as a reflected Brownian motion in a quarter plane}
\label{sec:3/4}

We now perform a change of variables to obtain a new process $\widetilde{Z}_t$ in the positive quarter plane, defined by
\begin{equation*}
   \widetilde{Z}_t \egaldef (-\widehat{Z}_t^1+\widehat{Z}_t^2,\widehat{Z}_t^2),
\end{equation*}
see Figure~\ref{fig:symmetric_case_change_variable}. This reformulation at hand, we will be able to use the numerous results in the literature on reflected Brownian motion in a quadrant. Let us emphasize here that our drift is vertical (as shown below), while most of the existing results actually assume that the drift is either zero or oblique (with two non-zero coordinates). Accordingly, some attention is needed when applying directly previous results.
\begin{prop}
The process $\widetilde{Z}_t$ satisfies
\begin{equation*}
\begin{cases}
\widetilde{Z}_t^1=\widetilde{Z}_0^1 +\widetilde{W}_t^1 +(1 - r) \widehat{L}_t^1 +  \widehat{L}_t^2,
\\
\widetilde{Z}_t^2=\widetilde{Z}_0^2+\widetilde{W}_t^2+ \mu t+ \widehat{L}_t^1 + \frac{1}{2} \widehat{L}_t^2,
\end{cases}
\end{equation*}
where $\widetilde{W}$ is a Brownian motion with covariance matrix
\begin{equation*}
\widetilde{\Sigma}\egaldef 
\left(
\begin{array}{cc}
\widetilde{\sigma}_1 &  \widetilde{\rho}   \\
 \widetilde{\rho}& \widetilde{\sigma}_2 
\end{array}
\right)
=
\left(
\begin{array}{cc}
2(\sigma -\rho) & (\sigma-\rho)   \\
 (\sigma-\rho)  & \sigma 
\end{array}
\right),
\end{equation*}
while $\widehat{L}_t^1$ is the local time of the process on the horizontal axis and $\widehat{L}_t^2$ is the local time on the vertical axis.
Thus $\widetilde{Z}_t$ is a reflected Brownian motion in the quadrant $\mathbb{R}_+^2$ with drift $(0,\mu)$, covariance matrix $\widetilde{\Sigma}$
and reflections $(\widetilde{r}_1,1)\egaldef (1-r,1)$ and $(1, \widetilde{r}_2 )\egaldef(1,1/2)$. 
\label{prop:Ztilde}
\end{prop}
\begin{proof}
By Lemma \ref{lem:semimart38}, we have
\begin{equation*}
\begin{cases}
\widehat{Z}_t^1=\widehat{Z}_0^1+ \widehat{W}_t^1+ \mu t + r \widehat{L}_t^1 - \frac{1}{2}\widehat{L}_t^2,
\\
\widehat{Z}_t^2=\widehat{Z}_0^2+ \widehat{W}_t^2+ \mu t+ \widehat{L}_t^1 +  \frac{1}{2}\widehat{L}_t^2,
\end{cases}
\end{equation*}
where $\widehat{W}_t$ is a Brownian motion with the same covariance matrix as $W_t$, $\widehat{L}_t^2$ is the local time of $\widehat{Z}_t$ on the diagonal and $\widehat{L}_t^1= L_t^1 + L_t^2$ is the local time of $\widehat{Z}_t$ on the horizontal axis. 
Then we have
\begin{equation*}
\begin{cases}
\widetilde{Z}_t^1=-\widehat{Z}_0^1+\widehat{Z}_0^2  - \widehat{W}_t^1+\widehat{W}_t^2 +(1 - r) \widehat{L}_t^1 +  \widehat{L}_t^2,
\\
\widetilde{Z}_t^2=\widehat{Z}_0^2+\widehat{W}_t^2+ \mu t+ \widehat{L}_t^1 + \frac{1}{2} \widehat{L}_t^2.
\end{cases}
\end{equation*}
The covariance matrix of the Brownian motion 
$
\widetilde{W}_t\egaldef (-\widehat{W}_t^1+\widehat{W}_t^2,\widehat{W}_t^2)
$
is 
\begin{equation*}
\left(
\begin{array}{cc}
\widetilde{\sigma}_1 &  \widetilde{\rho}   \\
 \widetilde{\rho}& \widetilde{\sigma}_2 
\end{array}
\right)
=
\left(
\begin{array}{cc}
2(\sigma -\rho) & (\sigma-\rho)   \\
 (\sigma-\rho)  & \sigma 
\end{array}
\right).\qedhere
\end{equation*}
\end{proof}

Let $\widehat{L}_{1}(x,y)$ be the Laplace transform of $\frac{1}{2}\widehat{\pi}$ and $\widetilde{L}_{1}(p,q)$ be the Laplace transform of $\frac{1}{2}\widetilde{\pi}$, where $\widetilde{\pi}$ is the stationary distribution of $\widetilde{Z}$. Let finally  $L_1(x,y)$  be the Laplace transform as in \eqref{eq:formula_Laplace_transform_L1}.
\begin{lem}
\label{lem:Laplacetransformchangevariable}
For $(p,q)=(-x,x+y)$, the various Laplace transforms satisfy
\begin{equation*} 
L_1(x,y)=\widehat{L}_{1}(x,y)=\widetilde{L}_{1}(-x,x+y)=\widetilde{L}_{1}(p,q).
\end{equation*}
\end{lem}
\begin{proof}
Proposition~\ref{prop:hatpi} implies that 
$
L_1(x,y)=\widehat{L}_{1}(x,y)$. Using that $\widetilde{Z}_t=(-\widehat{Z}_t^1+\widehat{Z}_t^2,\widehat{Z}_t^2)$, a simple change of variables in the Laplace transform yields $\widehat{L}_{1}(x,y)=\widetilde{L}_{1}(-x,x+y)$.
\end{proof}

\subsection{Functional equations}
We now state a functional equation, which characterizes the Laplace transform $\widetilde{L}_{1}(p,q)$.

\begin{prop}
\label{prop:sym_func_eq}
In the symmetrical case, the following functional equation holds:
\begin{equation}
\label{eq:funcsym2}
   U(p,q)\widetilde{L}_{1}(p,q) + C(p,q) \ell_1(p) +  A(p,q) m(q) = 0,
\end{equation}
where 
\begin{equation}
 \begin{cases}
\DD U(p,q) =  (\sigma-\rho)p^2 + (\sigma-\rho)qp +\frac{\sigma q^2}{2} + \mu q, \\[0.1cm]
 C(p,q) =  (1-r)p + q, \\[0.1cm]
A(p,q) =  (\sigma -\rho)(p+\frac{1}{2}q).
\end{cases}
\label{eq:coeffsymcase}
\end{equation}
\end{prop}
As a consequence of Proposition \ref{prop:sym_func_eq}, the Laplace transform $\widetilde{L}_{1}(p,q)$ may be computed along the same way as in~\cite{FrRa-19} (contour integral expressions) or \cite{BoElFrHaRa-21} (hypergeometric expressions). Interestingly, this functional equation may be obtained by two different techniques:
\begin{enumerate}[label=\arabic{*}.,ref=\arabic{*}]
   \item\label{proof1}We can use the functional equation~\eqref{eq:sys3} already obtained in the general (a priori non-symmetric) case 
and apply it to the symmetric case, using Lemma~\ref{lem:Laplacetransformchangevariable}.
   \item\label{proof2}We can also use Proposition~\ref{prop:Ztilde}, which says that $\widetilde{Z}$ is a reflected Brownian motion in a quadrant and use the functional equation already known in the bibliography \cite[Eq.~(2.3)]{DaMi-11} and \cite[Eq.~(5)]{FrRa-19}.
\end{enumerate}
We present both proofs below.
\begin{proof}[Proof \ref{proof1} (of Proposition \ref{prop:sym_func_eq})]
In the symmetric case, the main functional equation (see Proposition \ref{prop:main_func_eq}) takes the simpler form
\begin{equation}
\label{eq:funcsym1}
   K(x,y)L_1(x,y) + k (x,y) m (x+y) + k_1(x,y) \ell_1(x) = 0,
\end{equation}
where
\begin{equation*}
   K(x,y)= \frac{1}{2} (\sigma x^2  + 2\rho xy +\sigma y^2)+ \mu(x+y),
\end{equation*}
and 
\begin{equation*}
   k (x,y)=\frac{1}{2}(\sigma -\rho ) (-x+y) \quad \text{and} \quad k_1(x,y)=r x +y .
\end{equation*}
As in Section~\ref{sec:equafunc}, we introduce the new variables 
\begin{equation*}
   p=-x \quad \textrm{and} \quad q= x+y.
\end{equation*}
Keeping the same names for the unknown functions, we get from \eqref{eq:funcsym1} and~\eqref{eq:sys3}
\begin{equation*}
   U(p,q)L_1(p,q) + C(p,q) \ell_1(p) +  A(p,q) m(q) = 0,
\end{equation*}
where, by using~\eqref{eq:Ker1}, we obtain the value of $U$, $C$ and $A$ given in \eqref{eq:coeffsymcase}.
\end{proof}
\begin{proof}[Proof \ref{proof2} (of Proposition \ref{prop:sym_func_eq})]
By Proposition~\ref{prop:Ztilde}, the process $\widetilde{Z}$ is a reflected Brownian motion in a quadrant.  
We denote by $\widetilde{\nu}$ the density of the boundary invariant measure of $\widetilde{Z}$ on the vertical axis, which is defined by
\begin{equation*}
   \widetilde{\nu} (x){\mathrm{d}x} = \mathbb{E}_\Pi \int_0^1 \mathrm{1}_{\mathrm{d}x\times\{0\} } (\widetilde{Z}_s) \mathrm{d}\widehat{L}_s^2.
\end{equation*}
Now recall from \cite[\S 2.2]{BoElFrHaRa-21} that we have
\begin{equation*}
   \widetilde{\nu}(x) =(\sigma - \rho)\widetilde{\pi}(0,x)=2(\sigma - \rho)\pi(x,x).
\end{equation*}
It follows that the Laplace transform of $\widetilde{\nu}$ is equal to $2(\sigma -\rho)m(q)$.
It remains to use the well-known functional equation for a reflected Brownian motion in a quadrant, see, e.g., \cite[Eq.~(2.3)]{DaMi-11} and \cite[Eq.~(5)]{FrRa-19}.
Thus, we obtain the functional equation \eqref{eq:funcsym1}. 
\end{proof}

\subsection{The roots of the kernel $U(p,q)$}
The formulas of Lemmas~\ref{lem:PuPv} and~\ref{lem:QuQv} are simplified in a pleasant way.

\begin{lem} \label{lem:Q}
The function $U(p,q)$ in \eqref{eq:funcsym2}, viewed as a polynomial in the variable $q$, has  two roots $Q_1(p)$ and $Q_2(p)$, which are the branches of a two-sheeted covering over the $p$-plane. They are analytic in the whole complex plane cut along $(-\infty,p_1]\cup [p_2,\infty)$, with
\begin{equation}\label{eq:bpp}
   p_1= \frac{\mu\bigl(\sigma-\rho +\sqrt{2\sigma(\sigma-\rho)}\bigr)}{\sigma^2 - \rho^2}<0<
p_2= \frac{\mu\bigl(\sigma-\rho - \sqrt{2\sigma(\sigma-\rho)}\bigr)}{\sigma^2 - \rho^2}.
\end{equation}
The branches $Q_1(p)$ and $Q_2(p)$ are separated (except on the cut) and they satisfy 
\begin{equation}\label{eq:sep1}
\begin{cases}
\DD \Re(Q_1(ix)) \leq  0  \leq \Re(Q_2(ix)),  \quad  \forall x\in \mathbb R, \\[0.2cm]
\DD \Re(Q_1(p))   \leq  \Re(Q_2(p)), \quad \forall p \in \Cc.
\end{cases}
\end{equation}
\end{lem}

\begin{proof}
The last property of \eqref{eq:sep1} is a direct application of the maximum modulus principle to the function $\exp Q(p)$. The proof of the lemma is complete.
\end{proof}

Mutatis mutandis, the following lemma holds, with the convenient notation.
\begin{lem}
\label{lem:Psym}
The function $V(p,q)$, viewed as a polynomial in the variable $p$, has two roots $P_1(q)$ and $P_2(q)$, which are the branches of a two-sheeted covering over the $q$-plane. They are analytic in the whole complex plane cut along $(-\infty,q_1]\cup [q_2,\infty)$, with 
\begin{equation}
\label{eq:bpq}
   q_1= 0 < q_2 =  -\frac{4\mu}{\sigma+\rho}.
\end{equation}
They are separated and satisfy 
\begin{equation}
\label{eq:sep3}
\begin{cases}
\DD \Re(P_1(ix)) \leq 0 \leq \Re(P_2(ix)),  \quad  \forall x\in \mathbb{R}, \\[0.2cm]
\DD \Re(P_1(p))  \leq \Re(P_2(p)), \quad \forall p \in \Cc.
\end{cases}
\end{equation}
\end{lem}

With the above definitions, when $\mu<0$,
\begin{equation*}
   P_1(0)=P_2(0)=0 \quad \text{and} \quad Q_1(0)=0.
\end{equation*}
Our goal is to set a boundary value problem (BVP) for either of the functions $m(q)$ or $\ell_1(p)$ on an adequate hyperbola. 

\subsection{The hyperbolas}

The following lemma is an immediate application of the results of 
Lemma~\ref{lem:PuPv}.
\begin{lem}
\label{lem:hpq}
The functions $Q_1$ and $Q_2$ map the cut $(-\infty,p_1]$ 
(resp.\ $[p_2,\infty)$)
onto the right branch $\H_q^+$ (resp.\ the left branch  $\H_q^-$) of the hyperbola $\H_q$
\begin{equation}\label{eq:hq0}
   (\sigma+\rho) x^2 - (\sigma-\rho)y^2 + 4\mu x + \frac{2\mu^2}{\sigma} =0,
\end{equation}
rewritten in the canonical form (since $\sigma>\vert\rho\vert$) as 
\begin{equation}\label{eq:hq}
\left(x + \frac{2\mu}{\sigma+\rho}\right)^2  -\left(\frac{\sigma-\rho}{\sigma+\rho}\right) y^2 =
\frac{2\mu^2(\sigma-\rho)}{\sigma(\sigma+\rho)^2}.
\end{equation}
Similarly, $P_1$ and $P_2$ map the cut $(-\infty,q_1]$ (resp.\ $[q_2,\infty)$) 
onto  the right branch $\H_p^+$ (resp.\ the left branch  $\H_p^-$) of the hyperbola $\H_p$
\begin{equation}\label{eq:hp}
\left(x - \frac{\mu}{\sigma+\rho}\right)^2 - \left(\frac{\sigma-\rho}{\sigma+\rho}\right) y^2 =
\left(\frac{\mu}{\sigma+\rho}\right)^2,
\end{equation}
which goes through the point $(0,0)$.
\end{lem}

\begin{proof}
On the cuts $[p_1,\infty)$ and $(-\infty,p_2]$, the quantities $Q_1(p)$ and $Q_2(p)$ take complex conjugate values of the form  $x\pm iy$, where
\begin{align*}
   Q_1(p)+Q_2(p) & = \frac{-2[\mu+(\sigma-\rho)p]}{\sigma}=2x,\\
   Q_1(p)Q_2(p) & = \frac{2(\sigma-\rho)p^2}{\sigma}= x^2+y^2.
\end{align*}
Equations \eqref{eq:hq0} and \eqref{eq:hq} follow immediately, and \eqref{eq:hp} is obtained in an entirely similar way.
\end{proof}

\subsection{Analytic continuation and BVP} For any arbitrary simple closed curve $\U$, $G_\U$
(resp.\ $G^c_\U$) will denote the interior (resp.\ exterior) domain
bounded by $\U$, i.e., the domain remaining on the left-hand side when
$\U$ is traversed in the positive (counterclockwise) direction. This
definition remains valid for the case when $\U$ is unbounded but closable at infinity. For instance, $G_{\H_q^+}$ (resp.\ $G_{\H_q^+}^c$) is the region situated to the right (resp.\ to the left) of the branch $\H_q^+$ of the hyperbola $\H_q$.

\medskip
\begin{coro}\label{coro:automorphy} \mbox{ }
\begin{enumerate}[label=\arabic{*}.,ref=\arabic{*}]
   \item\label{it:1_compo}$G_{\H_p^- } \setminus [-\infty,p_1] \mathrel{{\underrightarrow{\ Q_2(p)\ }\atop
\overleftarrow{\ P_1(q)\ }}} G_{\H_q^+} \setminus [q_2,+\infty ]$ and the mappings are conformal.
  \item\label{it:2_compo}The values of $Q_1$ belong to $ G_{\H_q^+}^c$. 
  \item\label{it:3_compo}The values of $Q_2$ belong to $G_{\H_q^- }^c$. 
\end{enumerate}
Moreover, the following automorphy relationships hold:
\begin{eqnarray*}
P_1 \circ Q_1(p) &=& \left\{ \begin{array}{lllll} p, & \mbox{if} & p
\in G_{\H_p^+ }^c,  \\ \neq p, & \mbox{if} & p \in G_{\H_p^+ }. 
\end{array} \right.  \quad\mbox{Then } 
P_1 \circ Q_1(G_{\H_p^+ }^c) = G_{\H_p^+}^c \\
P_2 \circ Q_1(p) &=& \left\{ \begin{array}{lllll} p, & \mbox{if} & p
\in G_{\H_p^+ } , \\ \neq p, & \mbox{if} & p \in G_{\H_p^+}^c.
\end{array} \right. \quad\mbox{Then } 
P_2 \circ Q_1(G_{\H_p^+}) = G_{\H_p^+}.\\
P_1 \circ Q_2(p) &=& \left\{ \begin{array}{lllll} p, & \mbox{if} & p
\in G_{\H_p^-},  \\ \neq p, & \mbox{if} & p \in G_{\H_p^-}^c.  \\
\end{array} \right. \quad \mbox{Then }  P_1 \circ Q_2(G_{\H_p^-}) = G_{\H_p^-}. \\
P_2 \circ Q_2(p) &=& \left\{ \begin{array}{lllll} p, & \mbox{if} & p
\in G_{\H_p^-}^c, \\ \neq p & \mbox{if} & p \in G_{\H_p^-}.
\end{array} \right. \quad\mbox{Then }  P_2 \circ Q_2(G_{\H_p^-}^c) =G_{\H_p^-}^c\end{eqnarray*} 
\end{coro}
\begin{proof}
The arguments are  analogous to those presented in \cite[Chap.~5 and Chap.~6]{FIM-2017}. 
Assertion~\ref{it:1_compo} is immediate. As for assertions \ref{it:2_compo} and \ref{it:3_compo}, they  follow mainly from the maximum modulus principle applied to the functions $Q_1(p)$ and $Q_2(p)$ respectively. The automorphy relationships can be checked up to some tedious calculus (omitted). They also can be verified by using the following GeoGebra numerical animation \href{https://www.geogebra.org/m/phvjk35w}{https://www.geogebra.org/m/phvjk35w}
\end{proof}

Letting $q$ tend successively to the upper and lower edge of the
slit $(-\infty, q_1]$,  and using the fact that $m(q)$ is analytic in the left half-plane $\{\Re(q)\le0\}$, we  eliminate $m(q)$ from \eqref{eq:funcsym2} to get
\begin{equation}\label{eq:bvpsym} 
   \ell_1(P_1(q)) F(P_1(q),q) - \ell_1(P_2(q))F(P_2(q),q) = 0,\quad \mbox{for } q \in (-\infty, q_1],
\end{equation}
where 
\begin{equation*}
   F(p,q) = \frac{C(p,q)}{A(p,q)}.
\end{equation*}
Then the determination of $\ell_1(p)$, meromorphic in the 
domain  $G_{\H_p^+}^c$, is equivalent to solving a
BVP of Riemann-Hilbert-Carleman type, on the contour $\H_p^+$ in the
complex plane, as originally proposed in \cite{FaIa-79}. More precisely, by using the first two properties of Corollary~\ref{coro:automorphy}, and remembering that on the cut $(-\infty,q_1], P_1(q)=\overline{P_2(q)}$, this BVP takes the following form:
\begin{equation} \label{eq:bvpsym2}
\ell_1(p) K(p) - \ell_1(\overline{p}) K(\overline{p}) = 0, \quad p \in \H_p^+ ,
\end{equation}
where  $K(p) = F(p,Q_1(p))$, and $\ell_1$ is sought to be meromorphic inside $G_{\H_p^+}^c$,
its poles being the possible zeros of $C(p,Q_1(p))$ in the region 
$G_{\H_p^+}^c\cap\{\Re(p)>0\}$.

Interestingly, Corollary~\ref{coro:automorphy} allows to carry out the analytic continuation of the functions $\ell_1(p)$ and $m(q)$, satisfying  equation~\eqref{eq:funcsym2}. 
\begin{thm}
The functional equation
\begin{equation} \label{eq:prolong} 
\ell_1(p)F(p,Q_1(p)) - \ell_1(P_2\circ Q_1(p)) F(P_2\circ Q_1(p),Q_1(p)) = 0
\end{equation}
is valid for all~$p\in\Cc$ and provides the analytic continuation of $\ell_1$ as a
meromorphic function (the number of poles being finite) to the whole
complex plane cut along $[p_2,\infty)$. 
\end{thm}
\begin{proof}  It is a direct consequence of the automorphy properties given in Corollary~\ref{coro:automorphy}. Indeed, it suffices in equation~\eqref{eq:bvpsym} to let 
$q$ quit the cut $(-\infty,q_1]$, while remaining in $\H_q^-$. Then to this $q$ corresponds a point $p\in G_{\H_q^+}^c$ satisfying $P_1 \circ Q_1(p) =p$, which leads to equation~\eqref{eq:prolong}.
\end{proof}

\subsection{Reformulation as a reflected Brownian motion in a $\beta$-cone}
\label{sec:beta}

Let $\beta$ be the angle in $(\pi,2\pi)$ such that $\cos \beta = -\rho/\sigma$, that is
\begin{equation*}
   \beta =2\pi - \arccos (- \rho/\sigma) \in (\pi,2\pi).
\end{equation*}
The simple linear mapping 
\begin{equation*}
T\egaldef \frac{1}{\sqrt{\sigma}} \left( \begin{array}{cc}
\frac{1}{\sin \beta} & \cot \beta \\ 
0 & 1
\end{array} \right)
\end{equation*}
given in the appendix of \cite{FrRa-19} transforms the reflected Brownian motion $Z$ of covariance matrix $\Sigma$ in the three-quarter plane into a Brownian motion in a non-convex cone of angle $\beta$, with identity covariance matrix and with two equal reflection angles $\delta$ such that
\begin{equation}
\label{eq:deltadef}
   \tan {\delta} = \frac{\sin \beta}{r+\cos \beta}.
\end{equation}

\begin{prop}
The process $T \widetilde{Z}$ is a reflected Brownian motion in a cone of angle $\beta/2$ and reflection angle $\epsilon=\pi/2$ and $\delta\in(0,\pi)$ defined in~\eqref{eq:deltadef}, see Figure~\ref{fig:symmetric_case_change_variable}.
\end{prop}

\begin{proof}
The Brownian motion $\widetilde{Z}$ has the covariance matrix
\begin{equation*}
\left(
\begin{array}{cc}
\widetilde{\sigma}_1 &  \widetilde{\rho}   \\
 \widetilde{\rho}& \widetilde{\sigma}_2 
\end{array}
\right)
=
\left(
\begin{array}{cc}
2(\sigma -\rho) & (\sigma-\rho)   \\
 (\sigma-\rho)  & \sigma 
\end{array}
\right),
\end{equation*}
see Proposition~\eqref{prop:Ztilde}.  Let
\begin{equation*}
   \widetilde{\beta} 
   = \arccos \left( - \frac{\widetilde{\rho}}{ \sqrt{\widetilde{\sigma}_1 \widetilde{\sigma}_2}} \right)= \arccos \left(-\sqrt{\frac{1}{2}\left(1-\frac{\rho}{\sigma}\right)} \right)
\end{equation*}
the angle associated to the new kernel $U$. In particular, 
$\widetilde\beta\in(\frac{\pi}{2},\pi)$, and we have
\begin{equation*}
   \cos^2 \widetilde{\beta} = \frac{1+\cos \beta}{2},
\end{equation*}
whence
\begin{equation*}
 \cos \beta= \cos 2\widetilde{\beta} \quad \mathrm{and} \quad 
\beta =2\widetilde{\beta},
\end{equation*}
see also \cite[Lem.~10]{Tr-19} and \cite{Mu-19}.
Then the new reflection matrix is equal to
\begin{equation*}
\left(
\begin{array}{cc}
1 & \widetilde{r}_2 \\
\widetilde{r}_1 & 1
\end{array}
\right)
\egaldef
\left(
\begin{array}{cc}
1 & 1-r \\
1/2 & 1
\end{array}
\right).
\end{equation*}

Performing the same change of variables as in the appendix of \cite{FrRa-19}, this equation amounts to studying a Brownian motion in a wedge of angle $\widetilde{\beta}$, identity covariance matrix and reflection angles
\begin{equation*}
   \tan {\epsilon}= \frac{\sin \widetilde{\beta}}{ \widetilde{r}_1 \sqrt{\widetilde{\sigma}_1/\widetilde{\sigma}_2} +\cos \widetilde{\beta}}
\quad \text{and} \quad 
\tan {\delta}= \frac{\sin \widetilde{\beta}}{ \widetilde{r}_2 \sqrt{\widetilde{\sigma}_2/\widetilde{\sigma}_1} 
+\cos\widetilde{\beta}}.
\end{equation*}
Then we get
\begin{equation}
\label{eq:values_angle_hat}
   \tan {\epsilon}= \infty, \text{ i.e., }\epsilon=\pi/2
\quad \text{and} \quad
\tan {\delta}= \frac{2\cos \widetilde{\beta}  \sin \widetilde{\beta} }{r-1 +2\cos^2 \widetilde{\beta} } =  \frac{\sin \beta}{r+\cos \beta}.\qedhere
\end{equation}
\end{proof}

\subsection{Algebraic nature of the Laplace transform}
For reflected Brownian motion in a quadrant, the work \cite{BoElFrHaRa-21} proposes an exhaustive classification of the parameters (drift, opening of the cone and reflection angles), allowing to decide which of the following classes of functions the associated Laplace transform $\widetilde L_1(p,q)$ belongs to:
\begin{enumerate}[label={\rm (C\arabic{*})},ref={\rm (C\arabic{*})}]
   \item\label{it:class1}Rational
   \item\label{it:class2}Algebraic
   \item\label{it:class3}D-finite (D for Differentially) (by this, we mean that the Laplace transform satisfies two linear differential equations with coefficients in $\mathbb R(p,q)$, one in $p$ and one in~$q$)
   \item\label{it:class4}D-algebraic (that is, when it satisfies a polynomial differential equation in $p$, and another in $q$)
   \item\label{it:class5}D-transcendental (when it is non-D-algebraic)
\end{enumerate}
Notice that the classes \ref{it:class1} to \ref{it:class4} define a hierarchy, in the sense that
\begin{equation*}
   \ref{it:class1} \subset \ref{it:class2} \subset \ref{it:class3} \subset \ref{it:class4}.
\end{equation*}
A more probabilistic description of the models having a Laplace transform in the class \ref{it:class1} above is as follows:
\begin{itemize}
   \item The skew symmetric condition: $\epsilon+\delta=\pi$, which is a necessary and sufficient condition for the stationary distribution to be exponential, see \cite{HaWi-87b}. 
   \item The Dieker and Moriarty \cite{DiMo-09} criterion: $\epsilon+\delta-\pi \in -\beta \mathbb{N}$, which is a necessary and sufficient condition for the stationary distribution to be a sum of exponential terms.
\end{itemize}

Accordingly, we may transfer the classification of \cite{BoElFrHaRa-21} to our symmetric Brownian motion in a three-quarter plane, via its projection in the domain $S_1$ and its quadrant description $\widetilde{Z}$. Then the following proposion holds.

\begin{prop} \label{prop:application_class}
The Laplace transform of the reflected Brownian motion in the quarter plane $\widetilde{Z}$ is never rational (class \ref{it:class1}). However, there exist values of parameters such that $\widetilde{L}$ is D-algebraic, D-finite or algebraic.
\end{prop}
Before proving Proposition \ref{prop:application_class}, let us do some remarks:
\begin{itemize}
   \item As a consequence, there is no skew symmetry in the three-quarter plane (nor Dieker and Moriarty condition). From this point of view, Brownian motion in non-convex cones is deeply different from Brownian motion in convex cones.
   \item The above feature (absence of skew symmetry) admits a clear interpretation in terms of the growth of exponential functions in $\mathbb R^2$. Indeed, for $(a,b)\neq (0,0)$, an exponential function
   \begin{equation}
   \label{eq:formula_exp}
      (p,q)\mapsto \exp(-ap-bq)
   \end{equation}
   tends to infinity in half of the directions of $\mathbb R^2$, so such an exponential function (and any finite linear combination of exponential functions as well) will never be integrable on a non-convex domain. As a direct consequence, it cannot represent any stationary distribution.
   \item The example presented in Figure~\ref{fig:Franceschi_condition2} (see \eqref{eq:expression_remarkable_density}) has an algebraic Laplace transform, as computed in \cite{BoElFrHaRa-21}. It appears as the simplest example which one may construct in a non-convex wedge.
\end{itemize}

\begin{proof}[Proof of Proposition \ref{prop:application_class}]
The skew symmetric condition is
\begin{equation*}
2 \widetilde{\rho}  = \widetilde{r}_1  \widetilde{\sigma}_1 + \widetilde{r}_2  \widetilde{\sigma}_2,
\end{equation*}
or
\begin{equation*}
\sigma -\rho = (1/2)(\sigma - \rho) +(1-r) \sigma /2,
\end{equation*}
which yields
$
r=\rho /\sigma  <1
$.
Hence, as the recurrence conditions imply $r>1$, we can conclude that the skew symmetric case is not possible.
More generally the Dieker and Moriarty condition 
\begin{equation*}
\epsilon+\delta -\pi \in -\mathbb{N} \widetilde{\beta}
\end{equation*}
cannot hold, because $\epsilon+\delta -\pi= \delta -\pi/2 >0$. 
However, there exist some parameters such that
\begin{equation*}
   \pi/2+ {\delta} \in \widetilde{\beta} \mathbb{Z} + \pi \mathbb{Z} ,
\end{equation*}
which is exactly condition \cite{BoElFrHaRa-21} to admit a D-algebraic Laplace transform. 
\end{proof}

\subsection{Line of steepest descent of $\pi$}

In the symmetric case, we remarked that the Laplace transform of the normal derivative of $\pi$ along the diagonal is zero and then $n(x,y)=0$, see \eqref{eq:formula_Laplace_transform_n}. Thus we may formulate the following question, in the non-symmetric case: does there also exist a line (not necessarily the diagonal) along which the normal derivative of $\pi$ is zero?

Let us consider the steepest descent line of $\pi$ starting from $(0,0)$. In other words, we consider that $\pi$ is a potential and we are looking to the field line of $\text{grad}\;\pi$ passing through $(0,0)$. This defines the curve
\begin{equation*}
\mathcal{C}=\{ (z_1(t),z_2(t)): t\in\mathbb{R}_+ \},
\end{equation*}
where $(z_1(0),z_2(0))=(0,0)$ and
\begin{equation*}
\begin{cases}
\DD z_1'(t) = \frac{\partial \pi}{\partial z_1} (z_1(t),z_2(t)),
\\[0.3cm]
\DD z_2'(t) = \frac{\partial \pi}{\partial z_2} (z_1(t),z_2(t)).
\end{cases}
\end{equation*}
If we divide the three-quarter plane along this line, we obtain a functional equation with only two unknown functions. We focus on a few examples where the curve $\mathcal{C}$ is a simple half-line: 
\begin{itemize}
\item In the symmetric case studied in Section~\ref{sec:symmetric_case}, the curve $\mathcal{C}$ is simply the first diagonal.
\item In the special case of Figure~\ref{fig:Franceschi_condition2}, the curve $\mathcal{C}$ is the half-line starting from the origin and following the direction of the drift. 
\item In the quadrant, when the skew symmetric condition is satisfied, the stationary distribution has an exponential density of the form \eqref{eq:formula_exp} (up to a normalization constant), and the curve $\mathcal{C}$ is also a half-line of direction $-(a,b)$. 
\end{itemize}

\appendix

\section{Proof of Proposition \ref{prop:main_func_eq}}
\label{app:p}

\begin{proof}
 Let us introduce the three following sets
\begin{align*}
   S_1^\epsilon &\egaldef \{(z_1,z_2) : z_2>z_1 + \epsilon/\sqrt{2} \text{ and } z_2 \geq 0 \} , \\
   S_2^\epsilon&\egaldef\{ (z_1,z_2) : z_1>z_2 + \epsilon/\sqrt{2} \text{ and } z_1 \geq 0 \}
\end{align*}
and 
$
   \widetilde{S}^\epsilon \egaldef S \setminus (S_1^\epsilon \cup S_2^\epsilon).
$
Then, we define the function $I_\epsilon$ such that
\begin{equation}
\label{eq:definition_function_I}
   I_\epsilon(z_1,z_2) \egaldef
\begin{cases}
1 & \text{if } z\in S_1^\epsilon,
\\ 
\frac{z_2-z_1}{\sqrt{2}\epsilon} + \frac{1}{2} & \text{if } z\in \widetilde{S}^\epsilon,
\\
0 & \text{if } z\in S_2^\epsilon.
\end{cases}
\end{equation}
From now on, we will often omit to note the variables $(z_1,z_2)$. 
We have
\begin{equation*}
\nabla I_\epsilon =
\left(
\frac{\partial I_\epsilon}{\partial z_1}   , \frac{\partial I_\epsilon}{\partial z_2}  
 \right)
 =
\begin{cases}
(0,0) & \text{if } z\in S_1^\epsilon\cup S_2^\epsilon,
\\ 
\left(\frac{-1}{\sqrt{2}\epsilon} , \frac{1}{\sqrt{2}\epsilon} \right) & \text{if } z\in \widetilde{S}^\epsilon,
\end{cases}
\end{equation*}
and, for all $z\in S$,
\begin{equation}
\frac{\partial^2 I_\epsilon}{\partial z_1^2}  =
\frac{\partial^2 I_\epsilon}{\partial z_2^2}  =
-\frac{\partial^2 I_\epsilon}{\partial z_1\partial z_2} 
=
 \frac{1}{\sqrt{2} \epsilon} \bigl(\delta_{\epsilon/\sqrt{2}} (z_1-z_2) -\delta_{-\epsilon/\sqrt{2}} (z_1-z_2) \bigr),
 \label{eq:I''}
\end{equation}
where $\delta_a$ is the Dirac distribution at $a$. For the sake of brevity, we write
\begin{equation*}
I'_\epsilon \egaldef \frac{\partial I_\epsilon}{\partial z_2} = -\frac{\partial I_\epsilon}{\partial z_1}\qquad
\text{and}\qquad 
I''_\epsilon\egaldef \frac{\partial^2 I_\epsilon}{\partial z_1^2}  =
\frac{\partial^2 I_\epsilon}{\partial z_2^2}  = -\frac{\partial^2 I_\epsilon}{\partial z_1\partial z_2} .
\end{equation*}
Let us take 
$
f_\epsilon  \egaldef e^{x z_1+ y z_2} I_{\epsilon}.
$
Its first and second derivatives are equal to
\begin{align*}
   \frac{\partial f_\epsilon}{\partial z_1} &= \left(x I_\epsilon+\frac{\partial I_{\epsilon}}{\partial z_1}  \right)e^{x z_1+ y z_2},\\
   \frac{\partial f_\epsilon}{\partial z_2} &=\left(y I_\epsilon+\frac{\partial I_{\epsilon}}{\partial z_2}  \right)e^{x z_1+ y z_2},\\
   \frac{\partial^2 f_\epsilon}{\partial z_1^2} &=\left( x^2 I_{\epsilon}+2x \frac{\partial I_{\epsilon}}{\partial z_1} +\frac{\partial^2 I_{\epsilon}}{\partial z_1^2}  \right)e^{x z_1+ y z_2},\\
   \frac{\partial^2 f_\epsilon}{\partial z_2^2} &=\left( y^2 I_{\epsilon}+2y \frac{\partial I_{\epsilon}}{\partial z_2} +\frac{\partial^2 I_{\epsilon}}{\partial z_2^2}  \right)e^{x z_1+ y z_2},\\
   \frac{\partial f_\epsilon}{\partial z_1 \partial z_2} &=\left(xy I_{\epsilon}+x \frac{\partial I_{\epsilon}}{\partial z_2}+y \frac{\partial I_{\epsilon}}{\partial z_1}  + \frac{\partial^2 I_{\epsilon}}{\partial z_1\partial z_2}  \right)e^{x z_1+ y z_2} .
\end{align*}
Therefore, the generator at $f_\epsilon$ is given by
\begin{equation*}
\mathcal{G} f_\epsilon=
 \left( 
 K(x,y)I_{\epsilon} +\Bigl( \frac{\partial K}{\partial y}-\frac{\partial K}{\partial x}  \Bigr)   I'_{\epsilon}
 +
  \frac{1}{2}
\Bigl( 
 \frac{\partial^2 K}{\partial x^2} 
 +
 \frac{\partial^2 K}{\partial y^2} 
 -2 \frac{\partial^2 K}{\partial x\partial y} 
 \Bigr) I''_\epsilon 
 \right) e^{x z_1+ y z_2},
\end{equation*}
that is,
\begin{multline*}
\mathcal{G} f_\epsilon =\\
 \left( 
 K(x,y)I_{\epsilon} +\bigl( \sigma_2y - \sigma_1x -\rho(y-x) +\mu_2 -\mu_1 \bigr)   I'_{\epsilon} 
 +  \frac{1}{2} \left(\sigma_1+\sigma_2-2\rho\right) I''_\epsilon\right) e^{x z_1+ y z_2}.
\end{multline*}
We also have
\begin{align*}
R_1 \cdot \nabla f_\epsilon(z_1,0)&=
\bigl((r_1 x +y) I_{\epsilon} +(1-r_1) I_{\epsilon}' \bigr)
e^{x z_1},
\\
R_2 \cdot \nabla f_\epsilon(0,z_2)&=
\bigl(( x +r_2y) I_{\epsilon} +(r_2 -1) I_{\epsilon}' \bigr)
e^{y z_2}
.
\end{align*}
Now we apply the basic adjoint relationship of Proposition~\ref{prop:BAR} to $f_\epsilon$ (which can be written as the difference of two convex functions and therefore satisfies the hypotheses of  Proposition~\ref{prop:BAR}). Since all integrals converge, as $f_\epsilon$ and its derivatives are bounded in $S$ for all $(x,y)$ in $\{\Re{(x)}\geq0,\,\Re{(x+y)}\leq0\}$, we obtain
\begin{align}
\notag 0=K(x,y) &\int_S  I_{\epsilon}(z_1,z_2) e^{x z_1+ y z_2}  \pi(z_1,z_2)  \mathrm{d}z_1 \mathrm{d}z_2 
\\ \notag &+ 
\bigl( \sigma_2 y - \sigma_1 x - \rho (y-x) +\mu_2 -\mu_1 \bigr)
\int_S  I_{\epsilon}'(z_1,z_2)  e^{x z_1+ y z_2} \pi(z_1,z_2)  \mathrm{d}z_1 \mathrm{d}z_2 
\\ \notag
&+ \frac{1}{2}
\left( 
  \sigma_1 + \sigma_2 - 2\rho
 \right)
\int_S I_{\epsilon}''(z_1,z_2)  e^{x z_1+ y z_2} \pi(z_1,z_2)  \mathrm{d}z_1 \mathrm{d}z_2 
\\ \notag
&+ (r_1 x +y)\int_{-\infty}^0 
 I_{\epsilon}(z_1,0) e^{x z_1}\nu_1 (z_1) \mathrm{d}z_1 + 
 (1-r_1)\int_{-\infty}^0  I_{\epsilon}'(z_1,0)e^{x z_1} 
 \nu_1 (z_1) \mathrm{d}z_1
 \\
&+ (x + r_2y)\int_{-\infty}^0 
 I_{\epsilon}(0,z_2)e^{y z_1}\nu_2 (z_2) \mathrm{d}z_2 + 
 (r_2 -1)\int_{-\infty}^0  I_{\epsilon}'(0,z_2)e^{y z_2} 
 \nu_1 (z_2) \mathrm{d}z_2.
 \label{eq:barI}
\end{align}
Since $\underset{\epsilon\to 0}{\lim} I_{\epsilon} = \mathrm{1}_{S_1}$, the dominated convergence theorem implies that:
\begin{align*}
\underset{\epsilon\to 0}{\lim}\int_S I_{\epsilon}(z_1,z_2) e^{x z_1+ y z_2} \pi(z_1,z_2) \mathrm{d}z_1 \mathrm{d}z_2
&= \int_{S_1} e^{x z_1+ y z_2} \pi(z_1,z_2) \mathrm{d}z_1 \mathrm{d}z_2 = L(x,y),\\
\underset{\epsilon\to 0}{\lim}\int_{-\infty}^0 
 I_{\epsilon}(z_1,0)e^{x z_1}\nu_1 (z_1) \mathrm{d}z_1 
 &= \int_{-\infty}^0 e^{x z_1}
\nu_1 (z_1) \mathrm{d}z_1 =\ell_1(x),\\
\underset{\epsilon\to 0}{\lim}\int_{-\infty}^0 
 I_{\epsilon}(0,z_2)e^{y z_2}\nu_2 (z_2) \mathrm{d}z_2 
 &= 0.
\end{align*}
We also have $\underset{\epsilon\to 0}{\lim} I'_{\epsilon}(z_1,z_2) =  \delta_0 (z_2-z_1)$, then by continuity of $\pi$, $\nu_1$ and $\nu_2$, we obtain the limits:
\begin{align*}
\underset{\epsilon\to 0}{\lim}\int_S I'_{\epsilon}(z_1,z_2) e^{x z_1+ y z_2} \pi(z_1,z_2) \mathrm{d}z_1 \mathrm{d}z_2
&= \int_0^\infty e^{(x+y) z } \pi(z,z) \mathrm{d}z = m(x+y),\\
\underset{\epsilon\to 0}{\lim}\int_{-\infty}^0 
 I'_{\epsilon}(z_1,0)e^{x z_1}\nu_1 (z_1) \mathrm{d}z_1 
 &= \nu_1(0),\\
\underset{\epsilon\to 0}{\lim}\int_{-\infty}^0 
 I'_{\epsilon}(0,z_2)e^{y z_2}\nu_2 (z_2) \mathrm{d}z_2 
 &= \nu_2(0).
\end{align*}
Our next goal is to show that
\begin{align*}
\underset{\epsilon\to 0}{\lim}\int_S I''_{\epsilon}(z_1,z_2) e^{x z_1+ y z_2} \pi(z_1,z_2) \mathrm{d}z_1 \mathrm{d}z_2 = \frac{1}{2} n(x+y) + \frac{1}{2} (x-y) m(x+y) .
\end{align*}
To this end, we introduce the linear change of variables
\begin{equation*}
   (z_1,z_2)\egaldef \phi(u,v)=\left(\frac{u-v}{\sqrt{2}},\frac{u+v}{\sqrt{2}}\right),
\quad
(u,v)=\phi^{-1}(z_1,z_2)=
\left(\frac{z_1+z_2}{\sqrt{2}},\frac{z_2-z_1}{\sqrt{2}}\right),
\end{equation*} 
where $\det \phi =1$.
Recall that for arbitrary constants $a$ and $c$, we have $\delta_a (c\times  \cdot)=\frac{1}{|c|} \delta_{a/c} (\cdot)$. So we deduce from \eqref{eq:I''} the equality $I''_{\epsilon}\bigl(\phi(u,v)\bigr)=\frac{1}{2}\bigl(\delta_{-\frac{\epsilon}{2}} - \delta_{\frac{\epsilon}{2}} \bigr)(v)$.
Let us define 
\begin{equation*}
   g(u,v)\egaldef e^{x \frac{u-v}{\sqrt{2}}+ y \frac{u+v}{\sqrt{2}} } \pi\left(\frac{u-v}{\sqrt{2}},\frac{u+v}{\sqrt{2}}\right).
\end{equation*}
We have
\begin{align*}
\int_S I''_{\epsilon}(z_1,z_2) e^{x z_1+ y z_2} \pi(z_1,z_2) \mathrm{d}z_1  \mathrm{d}z_2
&= 
\frac{1}{2\epsilon} \int_{\phi^{-1}(S)}
 \left(\delta_{-\frac{\epsilon}{2}} - \delta_{\frac{\epsilon}{2}} \right)(v)\  g(u,v)  \mathrm{d}u  \mathrm{d}v
\\ &=
\frac{1}{2\epsilon} \int_{-\epsilon/2}^\infty
 \left( g\Bigl(u,-\frac{\epsilon}{2}\Bigr)
 - g\Bigl(u,\frac{\epsilon}{2}\Bigr) \right) \mathrm{d}u
\\ & \underset{\epsilon\to 0}{\longrightarrow} \frac{-1}{2}\int_0^\infty \frac{\partial g}{\partial v} (u,0) \mathrm{d}u
\\ &= \frac{1}{2} \int_0^\infty e^{(x+y) z}
 \left( \frac{\partial \pi}{\partial z_1} -\frac{\partial \pi}{\partial z_2}
 \right)(z,z)
 \mathrm{d}z 
 \\&\quad + \frac{1}{2} (x-y)\int_0^\infty e^{(x+y) z} \pi(z,z)\mathrm{d}z.
 \end{align*}
Finally, letting $\epsilon\to 0$ in \eqref{eq:barI} concludes the proof.
\end{proof}

Notice that in the proof of Proposition \ref{prop:main_func_eq}, the particular expression \eqref{eq:definition_function_I} is not at all crucial: any similar function with the desired properties would have been suitable.

\bibliographystyle{abbrv}
\bibliography{biblio_FFR}

\end{document}